\documentclass[]{elsarticle}
\usepackage{amsmath,amssymb,amsthm,mathtools,bbm,booktabs,array,tikz,pifont,comment,multirow,url,graphicx}
\newcommand{\bvec}[1]{\boldsymbol{#1}}
\newcommand{\dif}{{\rm d}}

\newlength{\overwdth}

\def\abs#1{\ensuremath{\left \lvert #1 \right \rvert}}
\newcommand{\bigabs}[1]{\ensuremath{\bigl \lvert #1 \bigr \rvert}}

\newcommand{\biggabs}[1]{\ensuremath{\biggl \lvert #1 \biggr \rvert}}
\newcommand{\norm}[2][{}]{\ensuremath{\left \lVert #2 \right \rVert}_{#1}}
\newcommand{\bignorm}[2][{}]{\ensuremath{\bigl \lVert #2 \bigr \rVert}_{#1}}

\DeclareMathOperator{\sgn}{sgn}

\DeclareMathOperator{\Prob}{Prob}

\DeclareMathOperator{\cost}{cost}
\DeclareMathOperator{\comp}{comp}

\DeclareMathOperator{\std}{std}

\DeclareMathOperator{\sign}{sign}

\DeclareMathOperator{\Order}{{\mathcal O}}

\newcommand{\vzero}{\bvec{0}}

\newcommand{\vL}{\bvec{L}}

\newcommand{\vx}{\bvec{x}}

\newcommand{\vy}{\bvec{y}}

\newcommand{\tfc}{\tilde{\fc}}

\newcommand{\tF}{\widetilde{F}}

\newcommand{\tcf}{\widetilde{\cf}}

\newcommand{\tn}{\tilde{n}}

\newcommand{\bbone}{\mathbbm{1}}
\newcommand{\naturals}{\mathbb{N}}
\newcommand{\reals}{\mathbb{R}}

\newcommand{\natzero}{\mathbb{N}_{0}}

\newcommand{\ca}{\mathcal{A}}
\newcommand{\cb}{\mathcal{B}}
\providecommand{\cc}{\mathcal{C}}

\newcommand{\cf}{\mathcal{F}}
\newcommand{\cg}{\mathcal{G}}

\newcommand{\ci}{\mathcal{I}}
\newcommand{\cj}{\mathcal{J}}

\newcommand{\cl}{\mathcal{L}}

\newcommand{\cn}{\mathcal{N}}

\newcommand{\cu}{\mathcal{U}}
\newcommand{\cv}{\mathcal{V}}
\newcommand{\cw}{\mathcal{W}}
\newcommand{\cx}{\mathcal{X}}

\newcommand{\fc}{\mathfrak{c}}
\newcommand{\fC}{\mathfrak{C}}

\DeclareMathOperator{\Var}{Var}
\DeclareMathOperator{\INT}{INT}
\DeclareMathOperator{\APP}{APP}
\DeclareMathOperator{\lin}{lin}

\DeclareMathOperator{\fix}{fix}
\DeclareMathOperator{\err}{err}
\DeclareMathOperator{\maxcost}{maxcost}
\DeclareMathOperator{\mincost}{mincost}
\newcommand{\herr}{\widehat{\err}}

\newtheorem{theorem}{Theorem}

\newtheorem{lem}{Lemma}
\newtheorem{cor}{Corollary}
\theoremstyle{definition}
\newtheorem{algo}{Algorithm}

\theoremstyle{remark}
\newtheorem{rem}{Remark}
\newcommand{\Fnorm}[1]{\abs{#1}_{\cf}}
\newcommand{\Ftnorm}[1]{\abs{#1}_{\tcf}}
\newcommand{\Gnorm}[1]{\norm[\cg]{#1}}

\journal{Journal of Complexity}

\begin{document}

\begin{frontmatter}

\title{The Cost of Deterministic, Adaptive, Automatic Algorithms:  Cones, Not Balls}
\author{Nicholas Clancy}
\author{Yuhan Ding} \ead{ding2@hawk.iit.edu}
\author{Caleb Hamilton}
\author{Fred J. Hickernell} \ead{hickernell@iit.edu}
\author{Yizhi Zhang} \ead{yzhang97@hawk.iit.edu}
\address{Room E1-208, Department of Applied Mathematics, Illinois Institute of Technology,\\ 10 W.\ 32$^{\text{nd}}$ St., Chicago, IL 60616}

\begin{abstract} 
Automatic numerical algorithms attempt to provide approximate solutions that differ from exact solutions by no more than a user-specified error tolerance. The computational cost is often determined \emph{adaptively} by the algorithm based on the function values sampled. While adaptive, automatic algorithms are widely used in practice, most lack \emph{guarantees}, i.e., conditions on input functions that ensure that the error tolerance is met. 

This article establishes a framework for guaranteed, adaptive, automatic algorithms. Sufficient conditions for success and two-sided bounds on the computational cost are provided in Theorems \ref{TwoStageDetermThm} and \ref{MultiStageThm}.  Lower bounds on the complexity of the problem are given in Theorem \ref{complowbd}, and conditions under which the proposed algorithms have optimal order are given in Corollary \ref{optimcor}. These general theorems are illustrated for univariate numerical integration and function recovery via adaptive algorithms based on linear splines.  

The key to these adaptive algorithms is performing the analysis for \emph{cones} of input functions rather than balls.  Cones provide a setting where adaption may be beneficial.
\end{abstract}

\end{frontmatter}

\section{Introduction}

Automatic algorithms conveniently determine the computational effort required to obtain an approximate answer that differs from the true answer by no more than an error tolerance, $\varepsilon$.  The required inputs are both $\varepsilon$ and a black-box routine that provides function values.  Unfortunately, most commonly used adaptive, automatic algorithms are not guaranteed to provide answers satisfying the error tolerance. On the other hand, most existing guaranteed automatic algorithms are not adaptive, i.e., they are do not adjust their effort based on information about the function obtained through sampling.  The goal here is to construct adaptive, automatic algorithms that are guaranteed to satisfy the error tolerance.

\subsection{Non-Adaptive, Automatic Algorithms for Balls of Input Functions} \label{nonadaptintrosubsec}
Let $\cf$ be a \emph{linear space} of input functions defined on $\cx$ with \emph{semi-norm} $\Fnorm{\cdot}$, let $\cg$ be a linear space of outputs with \emph{norm} $\norm[\cg]{\cdot}$, and let $S:\cf \to \cg$ be a \emph{solution operator}.  Suppose that one has a sequence of fixed-cost algorithms, $\{A_n\}_{n \in \ci}$, indexed by their computational cost, $n$, with $\ci \subseteq \natzero$.  Furthermore, suppose that there is some known error bound of the form
\begin{subequations} \label{algoerr}
\begin{equation} \label{traditionerra}
\norm[\cg]{S(f)-A_n(f)} \le h(n) \Fnorm{f},
\end{equation}
where $h:\ci \to [0,\infty)$ is non-negative valued and non-increasing. Note that $A_n$ must be exact for input functions with vanishing semi-norms, i.e., $S(f)=A_n(f)$ if $\Fnorm{f}=0$.   Furthermore, $h$ is assumed to have zero infimum, which makes it possible to define $h^{-1}$ for all positive numbers:
\begin{equation} \label{hinvdef}
\inf_{n \in \ci} h(n) = 0, \qquad h^{-1}(\varepsilon) = \min \{n \in \ci : h(n) \le \varepsilon\}, \qquad \varepsilon > 0.
\end{equation}
\end{subequations}
Error bound \eqref{algoerr} allows one to construct an automatic, yet non-adaptive, algorithm that is guaranteed for input functions in a prescribed $\cf$-ball.

\begin{algo}[Non-Adaptive, Automatic] \label{nonadaptalgo} Let $\{A_n\}_{n \in \ci}$ be defined as above, and let $\sigma$ be a fixed positive number.  For any input function $f\in \cb_{\sigma}:=\{ f \in \cf : \Fnorm{f} \le \sigma\}$ and any positive error tolerance $\varepsilon$, find the computational cost needed to satisfy the error tolerance, $n=h^{-1}(\varepsilon/\sigma)$.  Return $A_n(f)$ as the answer.
\end{algo}

\begin{theorem}  \label{NonAdaptDetermThm}  For $\cf$, $\Fnorm{\cdot}$, $\cg$, $\Gnorm{\cdot}$, $S$ as described above, and under the assumptions of Algorithm \ref{nonadaptalgo}, if $f$ lies in the ball $\cb_\sigma$, then the answer provided by Algorithm \ref{nonadaptalgo} must satisfy the error tolerance, i.e., $\norm[\cg]{S(f)-A_n(f)} \le \varepsilon$.
\end{theorem}

Algorithm \ref{nonadaptalgo}, Theorem \ref{NonAdaptDetermThm}, and the other theoretical results in this article related to Algorithm \ref{nonadaptalgo} are essentially known.  They serve as a benchmark to which we may compare our new adaptive algorithms.  

Algorithm \ref{nonadaptalgo} has drawbacks.  If it works for $f \in \cf$, it may not work for $cf \in \cf$, where $c>1$, because $cf$ may fall outside the ball $\cb_\sigma$.  Moreover, although error bound \eqref{traditionerra} depends on $\Fnorm{f}$, the computational cost of Algorithm \ref{nonadaptalgo} does not depend on $\Fnorm{f}$.  The cost is the same whether $\Fnorm{f}=\sigma$ or $\Fnorm{f}$ is much smaller than $\sigma$.  This is because Algorithm \ref{nonadaptalgo} is not adaptive.

\subsection{Adaptive, Automatic Algorithms for Cones of Input Functions} \label{adapintrosec}

Adaptive, automatic algorithms are common in numerical software packages.  Examples include  MATLAB's {\tt quad} and {\tt integral} \cite{MAT8.1}, the quadrature algorithms in the NAG Library \cite{NAG23}, and the MATLAB Chebfun toolbox \cite{TrefEtal12}.  While these adaptive algorithms work well for many cases, they have no rigorous justification. The methods used to determine the computational cost are either heuristics or asymptotic error estimates that do not hold for finite sample sizes.

In this article we derive guaranteed adaptive, automatic algorithms.  These adaptive algorithms use $\{A_n\}_{n \in \ci}$ with known $h$ as described in \eqref{algoerr} and satisfying some additional technical conditions in \eqref{algseqerrcond}.  Rather than assuming an upper bound on $\Fnorm{f}$, our adaptive algorithms use function data to construct \emph{rigorous} upper bounds on $\Fnorm{f}$.  We highlight the requirements here.

The key idea is to identify a suitable semi-norm on $\cf$, $\Ftnorm{\cdot}$, that is weaker than $\Fnorm{\cdot}$, i.e., there exists a positive constant $\tau_{\min}$ for which 
\begin{equation} \label{Fspacecondstrong}
\tau_{\min} \Ftnorm{f} \le \Fnorm{f} \qquad \forall f \in \cf.
\end{equation}
Moreover, there must exist a sequence of algorithms, $\{\tF_n\}_{n\in \ci}$, which approximates $\Ftnorm{\cdot}$ and has a two-sided error bound:
\begin{equation} \label{Gerrbds}
-h_{-}(n)\Fnorm{f} \le \Ftnorm{f}- \tF_n(f) \le h_{+}(n) \Fnorm{f}, \qquad \forall f \in \cf,
\end{equation}
for known non-negative valued, non-increasing $h_{\pm}$ satisfying $\inf_{n \in \ci} h_{\pm}(n) = 0$.  The adaptive algorithms to approximate $S$ are defined for a cone of input functions:
\begin{equation} \label{conedef}
\cc_{\tau}=\{f \in \cf : \Fnorm{f} \le \tau \Ftnorm{f} \}.
\end{equation}
(An arbitrary cone is a subset of a vector space that is closed under scalar multiplication.) Although the functions in this cone may have arbitrarily large $\tcf$- and $\cf$-semi-norms, the assumptions above make it
possible to construct reliable, data-driven upper bounds on $\Ftnorm{f}$ and $\Fnorm{f}$. 

The above assumptions are all that is required for our two-stage adaptive Algorithm \ref{twostagedetalgo}.  For our multi-stage adaptive Algorithm \ref{multistagealgo}, we further assume that the algorithms $\tF_n$ and $A_n$ use the same function data for all $n \in \ci$.  We also assume that there exists some $r>1$ such that for every $n \in \ci$ there exists an $\tn \in \ci$ satisfying $n < \tn \le rn$ and for which the data for $A_n$ are embedded in the data for $A_{\tn}$. One may think of $r$ as the cost multiple that one might need to incur when moving to the next more costly nested algorithm.

Section \ref{integsec} applies these ideas to the problem of evaluating $\int_0^1 f(x) \, \dif x$.  Here $\cf$ is the set of all continuous functions whose first derivatives have finite (total) variation, $\Fnorm{f}=\Var(f')$, and $\Ftnorm{f}=\norm[1]{f'-f(1)+f(0)}$.  The adaptive algorithm is a composite, equal-width, trapezoidal rule, where the number of trapezoids depends on the data-driven upper bound on $\Var(f')$.  The computational cost is no greater than $4+ \tau + \sqrt{\tau \Var(f')/(4\varepsilon)}$ (Theorem \ref{multistageintegthm}), where $\Var(f')$ is unknown. Here the cone constant $\tau$ is related to the minimum sample size, and $1/\tau$ represents a length scale for possible spikes that one wishes to integrate accurately.

\subsection{Scope and Outline of this Article} 
There are theoretical results providing conditions under which adaption is useful and when it is not useful. See for example, the comprehensive survey by Novak \cite{Nov96a} and more recent articles by Plaskota and Wasilkowski \cite{PlaWas05a,PlaEtal08a}. Here we consider a somewhat different situation.  Our focus is on cones of input functions because they provide a setting where adaptive stopping rules can be effective. Since adaptive stopping rules are often used in practice, even without theoretical guarantees, we want to justify their use.  However, the stopping rules that we adopt differ from those widely used (see Section \ref{Lynesssubsec}).

This article starts with the general setting and then moves to two concrete cases.  Section \ref{probdefsec} defines the problems to be solved and introduces our notation.  Sections \ref{genthmsec} and \ref{LowBoundSec} describe the adaptive algorithms in detail and provide proofs of their success for cones of input functions.  Our ultimate goal is to construct good locally adaptive algorithms, where the sampling density varies according to the function data.  However, here we present only globally adaptive algorithms, where the sampling density is constant, but the number of samples is determined adaptively.  Section \ref{integsec} illustrates the general results in Sections \ref{genthmsec} and \ref{LowBoundSec} for the univariate integration problem.  Section \ref{approxsec}  presents analogous results for function approximation.  Common concerns about adaptive algorithms are answered in Section \ref{overcomesec}. The article ends with several suggestions for future work.

\section{General Problem Definition} \label{probdefsec}

\subsection{Problems and Algorithms} The function approximation, integration, or other problem to be solved is defined by a solution operator $S:\cf \to \cg$ as described in Section \ref{nonadaptintrosubsec}. The solution operator is assumed to be positively homogeneous, i.e., 
\[
S(cf) = cS(f) \qquad \forall c\ge 0.
\]
Examples include the following:
\begin{align*}
\text{Integration:} \quad & S(f) = \int_{\cx} f(\vx) \, w(\vx) \, \dif \vx, \quad w \text{ is fixed,}\\
\text{Function Recovery:} \quad & S(f) = f, \\
\text{Poisson's Equation:} \quad & S(f) = u, \quad \text{where } \begin{array}{c} -\Delta u(\vx) = f(\vx), \ \vx \in \cx, \\ u(\vx)=0 \ \forall \vx \in \partial \cx, \text{ and}\end{array} \\
\text{Optimization:} \quad & S(f) = \min_{\vx \in \cx} f(\vx).
\end{align*}
The first three examples above are linear problems, but the last example is a nonlinear problem, which nevertheless is positively homogeneous.

Given a ``nice'' subset of input functions, $\cn \subseteq \cf$, an automatic algorithm $A:\cn\times(0,\infty) \to \cg$ takes as inputs a function, $f$, and an error tolerance, $\varepsilon$.  Our goal is to find an $A$ for which $\norm[\cg]{S(f)-A(f,\varepsilon)}\le \varepsilon$.  Algorithm \ref{nonadaptalgo} is one non-adaptive example that is successful for functions in balls, i.e., $\cn=\cb_{\sigma}$.

Following \cite[Section 3.2]{TraWasWoz88}, the algorithm takes the form of some function of data derived from the input function:
\begin{equation*}
\label{algoform}
A(f,\varepsilon) =  \phi(\vL(f)), \quad \vL(f) = \left(L_1(f), \ldots, L_m(f)\right) \qquad \forall f \in \cf.
\end{equation*}
Here the $L_i \in \Lambda$ are real-valued homogeneous functions defined on $\cf$:
\begin{equation*}
\label{dataassump}
L(cf) = cL(f) \qquad \forall f \in \cf, \ c \in \reals, \ L \in \Lambda.
\end{equation*}
One popular choice for $\Lambda$ is the set of all function values, $\Lambda^{\std}$, i.e., $L_i(f) = f(\vx_i)$ for some $\vx_i \in \cx$.  Another common choice is the set of all bounded linear functionals, $\Lambda^{\lin}$.  In general, $m$ may depend on  $\varepsilon$ and the $L_i(f)$, and each $L_i$ may depend on $L_1(f), \ldots, L_{i-1}(f)$.  The set of all such algorithms is denoted by $\ca(\cn,\cg,S,\Lambda)$. For example, Algorithm \ref{nonadaptalgo} lies in $\ca(\cb_{\sigma},\cg,S,\Lambda)$.  In this article, all algorithms are assumed to be deterministic.  There is no randomness.

\subsection{Costs of Algorithms} \label{AlgoCostsec}

The cost of a possibly adaptive algorithm, $A$, depends on the function and the error tolerance: 
\[
\cost(A,f,\varepsilon) = \$(\vL) = \$(L_1) + \cdots +\$(L_m) \in \natzero,
\]
where $\$:\Lambda \to \naturals$, and $\$(L)$ is the cost of acquiring the datum $L(f)$. The cost of $L$ may be the same for all $L \in \Lambda$, e.g, $\$(L)=1$.  Alternatively, the cost might vary with the choice of $L$.  For example, if $f$ is a function of the infinite sequence of real numbers, $(x_1, x_2, \ldots)$, the cost of evaluating the function with arbitrary values of the first $d$ coordinates, $L(f)=f(x_1, \ldots, x_d, 0, 0, \ldots)$, might be $d$.  This cost model has been used by for integration problems \cite{HicMGRitNiu09a,KuoEtal10a,NiuHic09a,NiuHic09b,PlaWas11a} and function approximation problems \cite{Was13a,WasWoz11a,WasWoz11b}.  If an algorithm does not require any function data, then its cost is zero.

Although the cost of an adaptive algorithm varies with $f$, we hope that it does not vary wildly for different input functions with the same $\cf$-semi-norm. We define the \emph{maximum} and \emph{minimum} costs of the algorithm $A\in \ca(\cn,\cg,S,\Lambda)$ relative to $\cb_{s}$, the $\cf$-semi-norm ball, as follows:
\begin{gather*}
\maxcost(A,\cn,\varepsilon,\cb_s)
= \sup \{ \cost(A,f,\varepsilon) : f \in \cn \cap \cb_{s} \}, \\ \mincost(A,\cn,\varepsilon,\cb_s)
= \inf \biggl \{ \cost(A,f,\varepsilon) : f \in \cn\setminus \bigcup_{0 \le s'<s}\cb_{s'} \biggr \} .
\end{gather*}
Note that $A$ knows that $f \in \cn$, but $A$ does not know $\Fnorm{f}$ (unless $\inf_{f \in \cn} \Fnorm{f}=\sup_{f \in \cn} \Fnorm{f}$).  An algorithm is said to have \emph{$\cb_{s}$-stable computational cost} if 
\begin{equation*}\label{coststabledef}
\sup_{\varepsilon, s > 0} \frac{\maxcost(A,\cn,\varepsilon,\cb_s)}{\max(1,\mincost(A,\cn,\varepsilon, \cb_s))} < \infty.
\end{equation*} 
An analogous definition of the stability of computational cost can be made in terms of $\tcf$-semi-norm balls.

The complexity of a problem is defined as the maximum cost of the cheapest algorithm that always satisfies the error tolerance:
\begin{multline*}
\comp(\varepsilon,\ca(\cn,\cg,S,\Lambda),\cb_s) \\
 = \inf\left\{\maxcost(A,\cn,\varepsilon,\cb_s) : A \in \ca(\cn,\cg,S,\Lambda), \right. \\
 \left . \norm[\cg]{S(f)-A(f,\varepsilon)} \le \varepsilon \ \ \forall f \in \cn, \ \varepsilon \ge 0 \right \} \in \natzero.
\end{multline*}
Here the infimum of an empty set is defined to be $\infty$.  

Algorithm \ref{nonadaptalgo} is defined for input functions lying in the ball $\cb_{\sigma}$.  It is not adaptive, and its cost depends only on $\varepsilon/\sigma$, but not on the particulars of $f$:
\begin{multline} \label{nonadaptalgocost}
\maxcost(A,\cb_{\sigma},\varepsilon,\cb_{s}) = \mincost(A,\cb_{\sigma},\varepsilon,\cb_{s}) \\
= \cost(A,f,\varepsilon) = h^{-1}(\varepsilon/\sigma) \qquad 0 < s \le \sigma.
\end{multline}

\subsection{Fixed-Cost Algorithms}

Automatic Algorithm \ref{nonadaptalgo} is built from a sequence of fixed-cost algorithms, $\{A_n\}_{n \in \ci}$.  The set of all fixed-cost algorithms is denoted by $\ca_{\fix}(\cf,\cg,S,\Lambda)$.  Any such algorithm is defined for all $f \in \cf$ and indexed by its cost.  Neither the number of function data nor the choice of the $L_i$ depend on the input function or $\varepsilon$, so we write $A_n(f)$ rather than $A_n(f,\varepsilon)$. Any fixed-cost algorithm is assumed be positively homogeneous:
\begin{equation*}
\label{algoscale}
\vL(cf) = c \vL(f), \ \
\phi(c\vy) = c\phi(\vy), \ \ A_n(cf) = cA_n(f) \quad \forall c \ge 0, \ f \in \cf,\ \vy \in \reals^m,
\end{equation*}
so its error, $\Gnorm{S(f)-A_n(f)}$, is positively homogeneous.

The adaptive algorithms in the next section use sequences of fixed-cost algorithms, $\{A_n\}_{n \in \ci}$ with $A_n  \in \ca_{\fix}(\cf,\cg,S,\Lambda)$ and indexed by their cost, $n=\cost(A_n)$.  The sequence $\{A_n(f)\}_{n \in \ci}$ converges to the true answer for all $f\in \cf$, as guaranteed by the conditions in \eqref{algoerr}.  Furthermore the index set, $\ci=\{N_1, N_2, \ldots\} \subseteq \natzero$, satisfies $N_i < N_{i+1}$ and
\begin{subequations} \label{algseqerrcond}
\begin{equation} \label{Inobiggap}
\sup_{i\ge 2} \frac{N_{i+1}}{N_i} \le \rho <\infty.
\end{equation}
Finally, in this article we assume that $h$ satisfies
\begin{equation} \label{hinvonetwocond}
\sup_{\epsilon > 0} \frac{h^{-1}(\varepsilon)}{\max(1,h^{-1}(2\varepsilon))} < \infty.
\end{equation}
\end{subequations}
This means that $h(n) = \Order(n^{-\alpha})$ as $n\to \infty$ for some $\alpha>0$.

\section{General Algorithms and Upper Bounds on the Complexity} \label{genthmsec}

This section provides general theorems about the cost of automatic algorithms.  The hypotheses of these theorems are non-trivial to verify for specific problems of interest.  However, the assumptions are reasonable as demonstrated by the examples in Sections \ref{integsec} and \ref{approxsec}.  

\subsection{Bounding the $\tcf$-Semi-Norm} \label{Galgosec}

As mentioned in Section \ref{adapintrosec}, adaptive, automatic algorithms require reliable upper bounds on $\Ftnorm{f}$ for all $f$ in the cone $\cc_{\tau}$. These can be obtained using any sequence of fixed-cost algorithms $\{\tF_n\}_{n \in \ci}$ with $\tF_n \in \ca_{\fix}(\cf,\reals_+,\Ftnorm{\cdot},\Lambda)$ satisfying the two-sided error bound in \eqref{Gerrbds}.  This implies that $\tF_n(f)=\Ftnorm{f}=0$ for all $f \in \cf$ with vanishing $\cf$-semi-norm.  Rearranging \eqref{Gerrbds} and applying the two bounds for the $\tcf$- and $\cf$-semi-norms in \eqref{Fspacecondstrong} and \eqref{conedef} implies that if $h_{ +}(n) < 1/\tau$, then for all $f \in \cc_{\tau}$,
\begin{gather*} 
\tF_n(f) \le \Ftnorm{f} + h_{-}(n)\Fnorm{f} \le \begin{cases} [1+ \tau h_{-}(n)]\Ftnorm{f} \\
\displaystyle \left [\frac{1}{\tau_{\min}}+h_{-}(n) \right]\Fnorm{f} 
\end{cases},\\
\tF_n(f) \ge \Ftnorm{f} - h_{+}(n)\Fnorm{f} \ge [1 - \tau h_{+}(n)]\Ftnorm{f}
\ge \left [\frac{1}{\tau}-h_{+}(n) \right]\Fnorm{f}.
\end{gather*}

\begin{lem} \label{Gnormlem} Any sequence of fixed-cost algorithms  $\{\tF_n\}_{n \in \ci}$ as described above with two sided error bound \eqref{Gerrbds} yields an approximation to the $\tcf$-semi-norm of functions in the cone $\cc_{\tau}$ with the following upper and lower bounds:
\begin{equation} \label{twosidedGineq}
 \frac{\Fnorm{f}}{\tau \fC_n}  \le \frac{\Ftnorm{f}}{\fC_n}  \le \tF_n(f) \le  \begin{cases} \tfc_n \Ftnorm{f} \\[1ex]
\displaystyle \frac{\fc_n \Fnorm{f} }{\tau_{\min}}
\end{cases} \qquad \forall f \in \cc_{\tau},
\end{equation}
where the $\fc_n$, $\tfc_n$, and $\fC_n$ are non-increasing in $n$ and defined as follows:
\begin{gather} \label{normdeflate}
\tfc_n :=1 + \tau h_{ -}(n)  \ge \fc_n :=1 + \tau_{\min} h_{ -}(n)  \ge 1, \\
\label{norminflate}
\fC_n :=\frac{1}{1 - \tau h_{ +}(n)}, \qquad  \fC_n \ge 1 \text{ for } h_{ +}(n) < 1/\tau.
\end{gather}
\end{lem}

\subsection{Two-Stage Adaptive Algorithms} \label{twostagesec}

Computing an approximate solution to the problem $S: \cc_{\tau} \to \cg$ also depends on a sequence of fixed-cost algorithms, $\{A_n\}_{n \in \ci}$, satisfying \eqref{algoerr} and \eqref{algseqerrcond}.  One may then use the upper bound in Lemma \ref{Gnormlem} to construct a data-driven upper bound on the error provided that $\fC_n>0$, i.e.,  $h_{ +}(n) < 1/\tau$:
\begin{equation} \label{Anerrbound}
\norm[\cg]{S(f) -  A_n(f)} \le h(n)\Fnorm{f} \le \tau  \fC_n h(n)\tF_n(f) \qquad \forall f \in \cc_{\tau}.
\end{equation}

\begin{algo}[Adaptive, Automatic, Two-Stage] \label{twostagedetalgo} Let $\tau$ be a fixed positive number, and let $\cc_\tau$ be the cone of functions defined in \eqref{conedef} whose $\cf$-semi-norms are no larger than $\tau$ times their $\tcf$-semi-norms.  Let $n_{\tF}$ satisfy $h_{+}(n_{\tF}) < 1/\tau$, and let $\tF_{n_{\tF}}$ be an algorithm as described in Lemma \ref{Gnormlem} with cost $n_{\tF}$.
Moreover, let  $\{A_n\}_{n \in \ci}$ be a sequence of algorithms as described in \eqref{algoerr} and \eqref{algseqerrcond}.  Given  a positive error tolerance, $\varepsilon$, and  an input function $f \in \cc_{\tau}$, do the following:

\begin{description} 

\item[Stage 1.\ Bound {$\Fnorm{f}$}.] First compute $\tF_{n_{\tF}}(f)$.  Define the inflation factor $\fC=\fC_{n_{\tF}}$ according to \eqref{norminflate}.
Then $\tau \fC \tF_{n_{\tF}}(f)$ is a reliable upper bound on $\Fnorm{f}$.  

\item [Stage 2.\ Estimate {$S(f)$}.] Choose the sample size needed to approximate $S(f)$, namely, $n_A=h^{-1}(\varepsilon/(\tau\fC \tF_{n_{\tF}}(f)))$.  Finally, return $A_{n_A}(f)$ as the approximation to $S(f)$ at a total cost of $n_{\tF}+n_A$. 
\end{description}
\end{algo}

The bounds in Lemma \ref{Gnormlem} involving $\tF_n$ imply bounds on the cost of the algorithm above.  Since $h^{-1}$ is non-increasing, it follows that for all $f \in \cc_{\tau}$,
\begin{equation*}
h^{-1}\left(\frac{\varepsilon}{\Fnorm{f}}\right) \le  h^{-1}\left(\frac{\varepsilon}{\tau\Ftnorm{f}}\right)  \le h^{-1}\left(\frac{\varepsilon}{\tau \fC \tF_{n_{\tF}}(f)}\right) \le 
\begin{cases} \displaystyle  h^{-1}\left(\frac{\varepsilon}{\tau \fC \tfc \Ftnorm{f}}\right) \\[2ex]
\displaystyle  h^{-1}\left(\frac{\tau_{\min} \varepsilon}{\tau \fC \fc \Fnorm{f}}\right)\end{cases}.
\end{equation*}

\begin{theorem}  \label{TwoStageDetermThm}  Let $\cf$, $\Fnorm{\cdot}$, $\Ftnorm{\cdot}$, $\cg$, $\Gnorm{\cdot}$, and $S$, and $\cc_\tau$ be as described above.  Under the assumptions of Algorithm \ref{twostagedetalgo}, let $\fc=\fc_{n_{\tF}}$ be defined as in \eqref{normdeflate}.  Then Algorithm \ref{twostagedetalgo}, which lies in $\ca(\cc_{\tau},\cg,S,\Lambda)$, is successful,  i.e.,  $\norm[\cg]{S(f)-A(f,\varepsilon)} \le \varepsilon$ for all $f \in \cc_\tau$.  Moreover, the cost of this algorithm is bounded above and below in terms of the unknown $\tcf$- and $\cf$-semi-norms of any input function in $\cc_{\tau}$ as follows:
\begin{multline}  \label{auto2stagedetcost}
n_{\tF}+ h^{-1}\left(\frac{\varepsilon}{\Fnorm{f}} \right)  \le n_{\tF}+ h^{-1}\left(\frac{\varepsilon}{\tau \Ftnorm{f}} \right) \\ \le 
\cost(A,f,\varepsilon)
\le \begin{cases}
\displaystyle n_{\tF}+ h^{-1}\left(\frac{\varepsilon}{\tau \fC\tfc\Ftnorm{f}}\right) \\[2ex]
\displaystyle n_{\tF}+h^{-1}\left(\frac{\tau_{\min}\varepsilon}{\tau \fC\fc\Fnorm{f}}\right)
\end{cases}.
\end{multline}
This algorithm is computationally stable in the sense that the maximum cost is no greater than some constant times the minimum cost, both for $\tcf$-balls and $\cf$-balls.
\end{theorem}

\begin{proof} The choice of $n_A$ in Algorithm \ref{twostagedetalgo} ensures that the right hand side of \eqref{Anerrbound} is no greater than the error tolerance, so the algorithm is successful, as claimed in the theorem.  The argument preceding this theorem establishes the two-sided cost bounds in \eqref{auto2stagedetcost}. The computational stability follows since $h$ satisfies \eqref{algseqerrcond}.
\end{proof}

There are several points to note about this result.

\begin{rem} This algorithm and its accompanying theorem assume the existence of fixed-cost algorithms for approximating the weaker semi-norm and for approximating the solution, both with known error bounds. Sections \ref{integsec} and \ref{approxsec} provide concrete examples where these conditions are satisfied.
\end{rem}

\begin{rem} The maximum and minimum costs of Algorithm \ref{twostagedetalgo} in \eqref{auto2stagedetcost} depend on the $\cf$- and $\tcf$-semi-norms of the input function, $f$.  However, the semi-norms of $f$ are not input to the algorithm, but rather are bounded by the algorithm.  The number of samples needed by Algorithm \ref{twostagedetalgo} is adjusted adaptively based on these bounds.
\end{rem}

\begin{rem}  Although non-adaptive Algorithm \ref{nonadaptalgo} and adaptive Algorithms \ref{TwoStageDetermThm} and \ref{MultiStageThm} are defined only for proper subsets of $\cf$, they may actually be applied to all $f\in \cf$ since the fixed-cost algorithms on which they are based are defined for all $f\in \cf$.  If the user unknowingly provides an input $f$ that does not belong to $\cb_\sigma$ for Algorithm \ref{nonadaptalgo} or $\cc_{\tau}$ for Algorithms \ref{TwoStageDetermThm} and \ref{MultiStageThm}, the answer returned may be wrong because the corresponding Theorem \ref{NonAdaptDetermThm}, \ref{TwoStageDetermThm}, or \ref{MultiStageThm} does not apply.
\end{rem}

\begin{rem} \label{neccondrem} In some cases it is possible to find a lower bound on the $\cf$-semi-norm of the input function, i.e., an algorithm $F_n$ using the same function values as $\tF_n$, such that
\[
F_n(f) \le \Fnorm{f} \qquad \forall f \in \cf.
\]
When such an $F_n$ is known, Lemma \ref{Gnormlem} can be used to derive a necessary condition that $f$ lies in the cone $\cc_\tau$:
\begin{align}
\nonumber
f \in \cc_\tau 
& \implies F_n(f) \le \Fnorm{f} \le \frac{\tau \tF_n(f)}{1-\tau h_+(n)}\\
& \implies \tau_{\min,n}:= \frac{F_n(f)}{\tF_n(f) + h_+(n) F_n(f)} \le \tau.
\label{neccondFn}
\end{align}
For Algorithm \ref{twostagedetalgo} the relevant value of $n$ is $n_{\tF}$, whereas for Algorithm \ref{multistagealgo} the relevant value of $n$ is $n_i$.
Condition \eqref{neccondFn} is not sufficient for $f$ to lie in $\cc_{\tau}$, so Algorithm \ref{twostagedetalgo} or \ref{multistagealgo} may yield an incorrect answer even if \eqref{neccondFn} is satisfied but $f \notin \cc_{\tau}$.  However, this argument suggests modifying Algorithms \ref{twostagedetalgo} and \ref{multistagealgo} by increasing $\tau$ to $2 \tau_{\min,n}$ whenever  $\tau_{\min,n}$ rises above $\tau$. 
\end{rem}

\begin{rem} \label{Nmaxrem}  For practical reasons one may impose a computational cost budget, $N_{\max}$.  If this is done, Algorithm \ref{twostagedetalgo} will compute the correct answer within budget for $f \in \cc_{\tau}$ if either of the cost upper bounds in Theorem \ref{TwoStageDetermThm} does not exceed $N_{\max}$.  An analogous result holds for Algorithm \ref{multistagealgo} and Theorem \ref{MultiStageThm}.
\end{rem}

\subsection{Adaptive Algorithms Based on Embedded Algorithms}

Suppose that $\{A_n\}_{n \in \ci}$, $A_n  \in \ca_{\fix}(\cf,\cg,S,\Lambda)$ now have the added property that some are embedded in others, as mentioned in Section \ref{adapintrosec}.  Let $r$ be the cost multiple described there.  Moreover, suppose that each $\tF_n$ uses the same data as $A_n$. These embedded algorithms suggest the following iterative adaptive algorithm.

\begin{algo}[Adaptive, Automatic, Multi-Stage] \label{multistagealgo}  Let the sequences of algorithms $\{A_n\}_{n \in \ci}$ and  $\{\tF_n\}_{n \in \ci}$ be as described above.  Let $\tau$ be the positive cone constant, and let $\cc_\tau$ be the cone of functions defined in \eqref{conedef} whose $\cf$-semi-norms are no larger than $\tau$ times their $\tcf$-semi-norms. Set $i=1$, and $n_1 = \min\{ n \in \ci : h_+(n) < 1/\tau\}$. For any positive error tolerance $\varepsilon$ and any input function $f$, do the following:
\begin{description}

\item [Stage 1. Estimate $\Ftnorm{f}$.] Compute $\tF_{n_i}(f)$ and $\fC_{n_i}$ as defined in \eqref{norminflate}.  

\item [Stage 2. Check for Convergence.] Check whether $n_i$ is large enough to satisfy the error tolerance, i.e., 
\begin{equation} \label{multistageconv}
\tau \fC_{n_i} h(n_i)\tF_{n_i}(f) \le \varepsilon.
\end{equation}
If this is true, return $A_{n_i}(f)$ and terminate the algorithm.

\item[Stage 3. Compute $n_{i+1}$.]  Otherwise, if \eqref{multistageconv} fails to hold, compute $\tfc_{n_i}$ according to \eqref{normdeflate}, and choose $n_{i+1}$ as the smallest number exceeding $n_i$ and not less than $h^{-1}(\varepsilon \tfc_{n_i}/[\tau \tF_{n_i}(f)])$ such that $A_{n_{i}}$ is embedded in $A_{n_{i+1}}$. Increment $i$ by $1$, and return to Stage 1.  

\end{description}  
\end{algo}

This iterative algorithm is guaranteed to converge also, and its cost can be bounded.  Define 
\begin{equation*} \label{hhdef}
h_1(n) := \fC_n\tfc_n h(n) \ge h(n), \quad h_2(n) := \fC_n\fc_n h(n) \ge h(n) \qquad n \in \ci,
\end{equation*}
and note that $h_1$ and $h_2$ are non-increasing functions.  Let $h_1^{-1}$ and $h_2^{-1}$ be defined analogously to $h^{-1}$ as in \eqref{hinvdef}.  These definitions imply that the quantity appearing on the left hand side in \eqref{multistageconv}, has the following upper and lower bounds based on \eqref{twosidedGineq} for $f \in \cc_{\tau}$:
\begin{subequations}\label{hh1h2bounds}
\begin{equation}
h(n)\Fnorm{f} \le \tau h(n)\Ftnorm{f} \le \tau \fC_{n} h(n)\tF_{n}(f) \le \begin{cases} \tau h_1(n) \Ftnorm{f} \\[1ex]
\displaystyle \frac{\tau h_2(n) \Fnorm{f} }{\tau_{\min}}
\end{cases} , \quad n \in \ci,
\end{equation}
\begin{multline}
h^{-1}\left(\frac{\varepsilon}{\Fnorm{f}}\right) \le h^{-1}\left(\frac{\varepsilon}{\tau \Ftnorm{f}}\right) \le \min\{n : \tau \fC_{n} h(n)\tF_{n}(f) \le \varepsilon\} \\
\le \begin{cases} \displaystyle h_1^{-1}\left(\frac{\varepsilon}{\tau \Ftnorm{f}} \right) \\[2ex]
\displaystyle h_2^{-1}\left(\frac{\tau_{\min}\varepsilon}{\tau \Fnorm{f}} \right)
\end{cases} , \qquad \varepsilon > 0.
\end{multline}
\end{subequations}
These inequalities may be used to prove the following theorem about Algorithm \ref{multistagealgo}, which is analogous to Theorem \ref{TwoStageDetermThm}.

\begin{theorem}  \label{MultiStageThm}  Let $\cf$, $\Fnorm{\cdot}$, $\Ftnorm{\cdot}$, $\cg$, $\Gnorm{\cdot}$, and $S$ be as described above.  Under the assumptions of Algorithm \ref{multistagealgo}, let $r$ be the cost multiple described in Section \ref{adapintrosec}.  
Then it follows that Algorithm \ref{multistagealgo}, which lies in $\ca(\cc_{\tau},\cg,S,\Lambda)$, is successful,  i.e.,  $\norm[\cg]{S(f)-A(f,\varepsilon)} \le \varepsilon$ for all $f \in \cc_\tau$.  Moreover, the cost of this algorithm is bounded above and below in terms of the unknown $\tcf$- and $\cf$-semi-norms of the input function as follows:
\begin{multline} \label{automultistagedetcost}
\max \left(n_1, h^{-1}\left(\frac{\varepsilon}{\tau \Ftnorm{f}} \right) \right) \le \max \left(n_1, h^{-1}\left(\frac{\varepsilon}{\Fnorm{f}} \right) \right) \\ \le
\cost(A,f,\varepsilon)
\le \begin{cases}\displaystyle \max \left(n_1, rh_1^{-1}\left(\frac{\varepsilon}{\tau \Ftnorm{f}} \right) \right)\\[2ex]
\displaystyle \max \left(n_1, rh_2^{-1}\left(\frac{\tau_{\min}\varepsilon}{\tau \Fnorm{f}} \right) \right)
\end{cases}.
\end{multline}
This algorithm is computationally stable in the sense that the maximum cost is no greater than some constant times the minimum cost, both for $\tcf$-balls and $\cf$-balls.
\end{theorem}

\begin{proof} Let $n_1, n_2, \ldots$ be the sequence of $n_i$ generated by Algorithm \ref{multistagealgo}.  We shall prove that the following statements must be true:
\begin{enumerate}
\renewcommand{\labelenumi}{\roman{enumi})}
\item If the convergence criterion \eqref{multistageconv} is satisfied for $i$, then the algorithm stops, $A_{n_i}(f)$ is returned as the answer, and it meets the error tolerance.

\item If the convergence criterion \eqref{multistageconv} is not satisfied for $i$, then $n_{i+1}$ does not exceed the cost upper bounds in \eqref{automultistagedetcost}.

\end{enumerate}
Statement i) holds because of the bounds in \eqref{Anerrbound} and in Lemma \ref{Gnormlem}.

If \eqref{multistageconv} is not satisfied for $i$, then it follows from  the inequality in \eqref{hh1h2bounds} that 
\[
n_{i} \le h_1^{-1}\left(\frac{\varepsilon}{\tau\Ftnorm{f}}\right) \qquad \text{and} \qquad 
n_{i} \le h_2^{-1}\left(\frac{\tau_{\min}\varepsilon}{\tau\Fnorm{f}}\right).
\] 
The algorithm then considers the candidate $n^*_{i+1} = h^{-1}(\varepsilon \tfc_{n_{i}}/[\tau \tF_{n_{i}}(f)])$ as a possible choice for $n_{i+1}$.  If $n^*_{i+1} \le n_{i}$, then $n^*_{i+1}$ is a bad choice for $n_{i+1}$, and Stage 3 chooses $n_{i+1}$ to be the smallest element of $\ci$ that exceeds $n_{i}$ and for which $A_{n_{i}}$ is embedded in $A_{n_{i+1}}$.  By the definition of $r$ it follows that $n_{i+1} \le r  n_{i}$, and so by the above inequalities for $n_{i}$, it follows that $n_{i+1}$ is bounded above by the right hand sides of the inequalities in \eqref{automultistagedetcost}.

If, on the other hand, $n^*_{i+1} > n_{i}$, then Stage 3 chooses $n_{i+1}$ to be the smallest element of $\ci$ that is no less than $n^*_{i+1}$ and for which $A_{n_{i+1}}$ is embedded in $A_{n_i}$.  By the definition of $r$, \eqref{twosidedGineq}, and the inequalities in \eqref{hh1h2bounds}, it follows that
\begin{equation*}
n_{i+1} < r n^*_{i+1} = r h^{-1}\left(\frac{\varepsilon \tfc_{n_{i}}}{\tau \tF_{n_{i}}(f)}\right) \le r  h^{-1}\left(\frac{\varepsilon}{\tau \Ftnorm{f}}\right) \le
\begin{cases}
\displaystyle  rh_1^{-1}\left(\frac{\varepsilon}{\tau \Ftnorm{f}} \right) \\[2ex]
\displaystyle rh_2^{-1}\left(\frac{\tau_{\min}\varepsilon}{\tau \Fnorm{f}} \right)
\end{cases}.
\end{equation*}
Again, $n_{i+1}$ is bounded above by the right hand sides of the inequalities in \eqref{automultistagedetcost}.

Since statement ii) now holds, the right hand side inequalities in \eqref{automultistagedetcost} also hold for cost of the algorithm.  The lower bounds on the computational cost follow from \eqref{hh1h2bounds}.  The computational stability follows since $h$ satisfies \eqref{algseqerrcond}.
\end{proof}

\section{Lower Complexity Bounds for the Problems} \label{LowBoundSec}
Lower complexity bounds are typically proved by constructing fooling functions.  Here we first derive a lower bound for the complexity of problems defined on an $\cf$-semi-norm ball of input functions, $\cb_\sigma$.  This technique is generally known, see for example \cite[p.\ 11--12]{TraWer98}.  Then it is shown how to extend this idea for the cone $\cc_{\tau}$. 

Let $\cj$ be a subset of $\natzero$.  Suppose that for any $n \in \cj$, and for all $\vL \in \Lambda^m$, satisfying $\$(\vL)\le n$, there exists an $f_1 \in \cf$, depending on $n$ and the $L_i$, with zero data, $\cf$-semi-norm smaller than one, and known lower bound on the solution, namely,
\begin{equation} \label{assumpfone}
\tau_{\min}\Ftnorm{f_1} \le \Fnorm{f_1} \le 1, \qquad \vL(f_1)= \vzero, \qquad
\norm[\cg]{S(f_1)} \ge g(n), 
\end{equation}
for some non-increasing function $g:\cj \to (0,+\infty)$ with $\inf_{n \in \cj} g(n)=0$.  For example, one might have $g(n)=a n^{-p}$ for $n \in \naturals$ with positive $a$ and $p$.

\subsection{Problems Defined on Balls} 
Suppose that $A$ is any successful automatic algorithm for the ball $\cb_{\sigma}$, i.e., $A \in \ca(\cb_{\sigma},\cg,S,\Lambda)$, and $\Gnorm{S(f)-A(f)}\le \varepsilon$ for all $f \in \cb_{\sigma}$.  For any fixed $s \le \sigma$ and $\varepsilon>0$, let $\vL$ be the design used by $A$ for the zero function.  Let $N_{j}$ and $N_{j+1}$ be two successive elements of $\cj$ with $N_{j}+1 \le \$(\vL) \le N_{j+1}$.

Let $f_1$ be constructed according to \eqref{assumpfone} for this $\vL$.  Since the data for the functions $\pm s f_1$ are all zero, it follows that $A(s f_1)=A(-s f_1)$.  Also note that $\pm s f_1 \in \cb_{s}$.  Since $A$ must be successful for $\pm s f_1$, it follows that 
\begin{align*}
\varepsilon & \ge \max(\norm[\cg]{S(sf_1)-A(sf_1)},\norm[\cg]{S(-sf_1)-A(-sf_1)}) \\
& \ge \frac{1}{2} \left[ \norm[\cg]{S(sf_1)-A(sf_1)}+ \norm[\cg]{S(sf_1)+A(sf_1)} \right] \\
& \ge \frac{1}{2} \norm[\cg]{[S(sf_1)-A(sf_1)]+[S(sf_1)+A(sf_1)]} \\
& = \norm[\cg]{S(sf_1)}= s\norm[\cg]{S(f_1)}  \\
& \ge s g(N_{j+1}).
\end{align*}
Since $\pm s f_1 \in \cb_s$, it follows that $N_{j}+1 \le \$(\vL) \le \maxcost(A,\cb_{\sigma},\varepsilon,\cb_{s})$.  The inequality $g(N_{j+1}) \le \varepsilon/s$ implies that $N_{j}$ can be no smaller than the largest $n \in \cj$ with $g(n) > \varepsilon/s$.  Thus, $g^{-1}(\varepsilon/s) \le N_{j}+1 \le \maxcost(A,\cb_{\sigma},\varepsilon,\cb_{s})$, where $g^{-1}$ is defined by
\begin{equation*} \label{ginvdef}
g^{-1}(\varepsilon) = \max \{ n \in \cj : g(n) > \varepsilon\}+1.
\end{equation*}
Here the maximum of the empty set is assumed to be $-1$.

\begin{theorem} \label{complowbdball} The computational complexity of the problem for a ball of input functions $\cb_{\sigma}$ is bounded below by
\begin{equation*}
\comp(\varepsilon,\ca(\cb_{\sigma},\cg,S,\Lambda),\cb_s) \ge
g ^{-1}(\varepsilon/\min(\sigma,s)).
\end{equation*}
\end{theorem}

The lower bound in this theorem and the upper bound in Theorem \ref{NonAdaptDetermThm} lead to a simple condition that guarantees the optimality of Algorithm \ref{nonadaptalgo}.  One only need to look at the ratio $h(n)/g(n)$.

\begin{theorem} \label{optimalprop} Suppose that there exist fixed-cost algorithms $\{A_n\}_{n \in \ci}$, with $A_n  \in \ca_{\fix}(\cf,\cg,S,\Lambda)$, for which the upper error bounds satisfy \eqref{algoerr} for known $h$ defined on $\ci$, which satisfies \eqref{Inobiggap}.  Suppose also that $g$ is defined on $\cj \supseteq \ci$. If 
\begin{equation}\label{ghcond}
\sup_{n \in \ci} \frac{h(n)}{g(n)} < \infty,
\end{equation}
then Algorithm \ref{nonadaptalgo} has optimal order in the sense that for fixed $\sigma$ and $s$, 
\begin{equation*} \label{balloptresult}
\sup_{\varepsilon > 0} \frac{\maxcost(A,\cb_{\sigma},\varepsilon,\cb_{s})}
{\max(1,\comp(\varepsilon,\ca(\cb_{\sigma},\cg,S,\Lambda),\cb_s))} <\infty.
\end{equation*}
\end{theorem}

\begin{proof}  Choose a number $C \ge \sup_{n \in \ci} h(n)/g(n)$.  It then follows that
\begin{align*}
g^{-1}(\varepsilon) &= \max \{ n \in \cj : g(n) > \varepsilon\}+1 \\ 
& \ge \frac{1}{\rho}\min\{ n \in \ci : g(n) \le \varepsilon\}  \\
& \ge \frac{1}{\rho} \min \left \{ n \in \ci : \frac{h(n)}{C} \le \varepsilon \right\} = \frac{h^{-1}(C \varepsilon)}{\rho}.
\end{align*}
Thus, it follows from the expression for the maximum cost in \eqref{nonadaptalgocost} and the condition on $h$ in \eqref{hinvonetwocond} that the error of Algorithm \ref{nonadaptalgo} is no worse than a constant times the best possible algorithm:
\begin{align*}
\MoveEqLeft{\sup_{\varepsilon \ge 0} \frac{\maxcost(A,\cb_{\sigma},\varepsilon,\cb_{s})}
{\max(1,\comp(\varepsilon,\ca(\cb_{\sigma},\cg,S,\Lambda),\cb_s))}} \\
& \le \sup_{\varepsilon \ge 0} \frac{h^{-1}(\varepsilon/\sigma)} {\max(1,g^{-1}(\varepsilon/\min(\sigma,s)))} \\
& \le \sup_{\varepsilon \ge 0} \frac{h^{-1}(\varepsilon/\sigma)} {\max(1,\rho^{-1} h^{-1}(C\varepsilon/\min(\sigma,s)))} <\infty.
\end{align*}
\end{proof}

\subsection{Problems Defined on Cones} \label{coneoptsubsec}

Now we turn to solving the numerical problem where the input functions lie in the cone $\cc_{\tau}$.  The cone condition makes the complexity lower bound more challenging to derive.  Moreover, we must now assume that the solution operator $S$ is \emph{linear}.  Condition \eqref{assumpfone} does not require the fooling function $f_1$ to lie \emph{inside} this cone.  To remedy this defect, fooling functions are constructed as linear combinations of $f_1$ and another function, $f_0$, lying in the interior of the cone.  Specifically, $f_0$ is assumed to satisfy
\begin{equation}
\label{assumpfzero}
\Ftnorm{f_{0}} = 1, \qquad \Fnorm{f_{0}} \le \tau_{\min} \Ftnorm{f_{0}} = \tau_{\min} < \tau,
\end{equation}
where $\tau_{\min}$ is defined in \eqref{Fspacecondstrong}.

\begin{theorem} \label{complowbd} Let $S$ be linear.  Assume that $\tau > \tau_{\min}$ and let $s$ be some positive number.  Suppose that functions $f_{0}$ and $f_1$ can be found that satisfy conditions \eqref{assumpfone} and \eqref{assumpfzero}.  It then follows that the complexity of the problem for cones of input functions is bounded below by
\begin{equation*}
\comp(\varepsilon,\ca(\cc_{\tau},\cg,S,\Lambda),\cb_{s}) 
\ge g^{-1}\left(\frac{2 \tau \varepsilon}{s(\tau-\tau_{\min})}\right).
\end{equation*}
\end{theorem}

\begin{proof} Let $A \in \ca(\cc_{\tau},\cg,S,\Lambda)$ be an arbitrary successful, possibly adaptive, algorithm.  Given an error tolerance, $\varepsilon$, and a positive $s$, let $f_0$ be a function satisfying \eqref{assumpfzero}, and choose 
\begin{subequations}\label{c0c1bumpdef}
\begin{equation} 
\label{c0bumpdef}
c_0 = \frac{s(\tau+\tau_{\min})}{2\tau\tau_{\min}} > 0.
\end{equation} 
Provide the algorithm $A$ with the input function $c_0f_0$, and let $\vL(c_0f_0)$ be the data vector extracted by $A$ to obtain the estimate $A(c_0f_0)$. Let $\$(\vL)$ denote the cost of this algorithm for the function $c_0f_0$, and let $N_j, N_{j+1} \in \cj$ be chosen as before such that $N_{j}+1 \le \$(\vL) \le N_{j+1}$.  Define two fooling functions, $f_{\pm}=c_0f_{0} \pm c_1 f_1$, in terms of $f_1$ satisfying conditions \eqref{assumpfone} with
\begin{equation} 
\label{c1bumpdef}
c_1 = \frac{s (\tau-\tau_{\min})}{2\tau } > 0.
\end{equation}
\end{subequations}
Both fooling functions have $\cf$-semi-norms no greater than $s$, since
\begin{align*}
\Fnorm{f_{\pm}} &\le c_{0} \Fnorm{f_0} + c_1 \Fnorm{f_1} \\
& = \frac{s}{2\tau}\left[\frac{\tau+\tau_{\min}}{\tau_{\min}}\tau_{\min} + (\tau-\tau_{\min}) \right]=s \qquad \text{by \eqref{c0c1bumpdef}}.
\end{align*}
Moreover, these fooling functions must lie inside the cone $\cc_{\tau}$ because
\begin{align*}
\Fnorm{f_{\pm}} - \tau  \Ftnorm{f_{\pm}} 
& \le  s - \tau (c_{0}\Ftnorm{f_0} - c_1 \Ftnorm{f_1}) \\
& \qquad \qquad \qquad \qquad \qquad \qquad \text{by the triangle inequality} \\
& \le  s - \tau c_{0} + \frac{\tau}{\tau_{\min}} c_1 \qquad \text{by \eqref{assumpfone}, \eqref{assumpfzero}}\\
& =  s - \frac{s(\tau+\tau_{\min})}{2\tau_{\min}} + \frac{s (\tau-\tau_{\min})}{2\tau_{\min}}= 0 \qquad \text{by \eqref{c0c1bumpdef}}.
\end{align*}

Following the argument earlier in this section, we note that the data used by algorithm $A$ for both fooling functions is the same, i.e., $\vL(f_{\pm})=\vL(c_0f_0)$, and so $A(f_{\pm})=A(c_0f_0)$.  Consequently, by the same argument used above, 
\[
\varepsilon  \ge  \max(\norm[\cg]{S(f_+)-A(f_+)},\norm[\cg]{S(f_-)-A(f_-)}) \ge c_1 \norm[\cg]{S(f_1)}\ge c_1 g(N_{j+1}).
\]
Here we have used the fact that $S$ is linear.  Since $A$ is successful for these two fooling functions, it follows that $\$(\vL)$, the cost of this arbitrary algorithm, is no greater than $\maxcost(A,\cc_{\tau},\varepsilon,\cb_s)$ and is bounded below by
\[
g^{-1} \left ( \frac{\varepsilon}{c_1} \right ) = g^{-1}\left(\frac{2 \tau \varepsilon}{s(\tau-\tau_{\min})}\right).
\]
This then implies the lower bound on the complexity of the problem.   
\end{proof}

\begin{cor} \label{optimcor}
Suppose that the functions $g$ and $h$ satisfy the hypotheses of Theorem \ref{optimalprop}, and in particular, condition \eqref{ghcond}, which means that Algorithm \ref{nonadaptalgo} has optimal order for solving the problem on $\cf$-balls of input functions.  It then follows that Algorithms \ref{twostagedetalgo} and \ref{multistagealgo} both have optimal order for solving the problem on for input functions lying in the cone $\cc_{\tau}$ in the sense of
\begin{equation*} \label{nearoptdef}
\sup_{\varepsilon,s > 0} \frac{\maxcost(A,\cc_\tau,\varepsilon,\cb_s)} {\max(1,\comp(\varepsilon,\ca(\cc_\tau,\cg,S,\Lambda),\cb_s))} <\infty.
\end{equation*} 
\end{cor}

\begin{table}
\centering
\begin{tabular}{>{\centering}m{6cm}>{\centering}m{5cm}}
\toprule
{\bf Minimum} cost of the {\bf best} algorithm that knows $f \in \cb_{\sigma}$ & $\displaystyle \ge g^{-1}\left(\frac{ \varepsilon}{\min(\sigma,\Fnorm{f})}\right)$ \tabularnewline
\midrule
{\bf Minimum} cost of the {\bf best} algorithm that knows that $f \in \cc_{\tau}$ & $\displaystyle \ge g^{-1}\left(\frac{2 \tau \varepsilon}{\Fnorm{f}(\tau-\tau_{\min})}\right)$ \tabularnewline
\midrule
Cost of {\bf non-adaptive} Algorithm \ref{nonadaptalgo} that knows $f \in \cb_{\sigma}$ & $\displaystyle h^{-1}\left(\frac{\varepsilon}{\sigma}\right)$ \tabularnewline
\midrule
{\bf Minimum} cost of {\bf adaptive} Algorithm \ref{multistagealgo} that knows $f \in \cc_{\tau}$ & $\displaystyle \ge \max\left( n_1, h^{-1}\left(\frac{\varepsilon}{\Fnorm{f}}\right)\right)$ \tabularnewline
\midrule
{\bf Maximum} cost of {\bf adaptive} Algorithm \ref{multistagealgo} that knows $f \in \cc_{\tau}$ & $\displaystyle \le \max \left(n_1, rh_2^{-1}\left(\frac{\tau_{\min}\varepsilon}{\tau \Fnorm{f}} \right) \right)$ \tabularnewline
\bottomrule
\end{tabular}
\caption{Costs of various algorithms, $A$, guaranteed to satisfy the tolerance, i.e.,  $\norm[\cg]{S(f) - A(f)} \le \varepsilon$. In all cases $\Fnorm{f}$ is unknown to the algorithm. \label{costcomparefig}}
\end{table}

Table \ref{costcomparefig} summarizes the lower and upper bounds on the computational cost of computing $S(f)$ to within an absolute error of $\varepsilon$.  The results summarized here are based on Theorems \ref{NonAdaptDetermThm}, \ref{MultiStageThm}, \ref{complowbdball}, and \ref{complowbd}.  Under condition \eqref{ghcond} all the algorithms mentioned in Table \ref{costcomparefig} have roughly the same computational cost.  

However, in the limit of vanishing $\varepsilon$ and $\Fnorm{f}$ with $\varepsilon/\Fnorm{f}$ held constant, the non-adaptive Algorithm \ref{nonadaptalgo} has unbounded cost, while the adaptive Algorithm \ref{multistagealgo} has bounded cost. The disadvantage of the non-adaptive algorithm is also seen in the two optimality results.  The supremum in Corollary \ref{optimcor} is taken over $s$ as well as $\varepsilon$, whereas the supremum in Theorem \ref{optimalprop} can only be taken over $\varepsilon$.

The next two sections illustrate the results of Section \ref{genthmsec} and \ref{LowBoundSec} for the problems of integration and approximation. Algorithm \ref{multistagealgo} is given explicitly for these two cases along with the guarantees provided by Theorem \ref{MultiStageThm}, the lower bound on complexity provided by Theorem \ref{complowbd}, and the optimality given by Corollary \ref{optimcor}.

\section{Approximation of One-Dimensional Integrals} \label{integsec}

The algorithms used in this section on integration and the next section on function recovery are all based on linear splines on $[0,1]$.  The node set and the linear spline algorithm using $n$ function values are defined for $n \in \mathcal{I}:=\{2,3,\ldots\}$ as follows:
\begin{subequations} \label{linearspline}
\begin{equation}
x_i=\frac{i-1}{n-1}, \qquad i=1, \ldots, n,
\end{equation}
\begin{multline}
A_{n}(f)(x):=(n-1) \left[ f(x_{i})(x_{i+1}-x) +f(x_{i+1})(x-x_i) \right] \\ \text{for }x_i \leq x \leq x_{i+1}.
\end{multline}
\end{subequations}
The cost of each function value is one and so the cost of  $A_n$ is $n$. The algorithm $A_n$ is imbedded in the algorithm $A_{2n-1}$, which uses $2n-2$ subintervals.  Thus, $r=2$ is the cost multiple as described in Section \ref{adapintrosec}.

The problem to be solved is univariate integration on the unit interval, $S(f):=\INT(f):=\int_{0}^{1}f(x) \, \dif x \in \cg := \reals$.  The fixed cost building blocks to construct the adaptive integration algorithm are the composite trapezoidal rules based on $n-1$ trapezoids:
\begin{equation*}
    T_{n}(f) := \int_0^1 A_n(f) \, \dif x
    =\frac{1}{2n-2}[f(x_1)+2f(x_2)+\cdots+2f(x_{n-1})+f(x_n)].
\end{equation*}

The space of input functions is $\cf:=\mathcal{V}^{1}$, the space of functions whose first derivatives have finite variation.  The general definitions of some relevant norms and spaces are as follows:
\begin{subequations} \label{defSobolev}
\begin{gather}
\Var(f) := \sup_{\substack{n \in \naturals\\ 0 = x_0 < x_1 < \cdots < x_{n} =1}} \sum_{i=1}^n \abs{f(x_i)-f(x_{i-1})}, \\
\norm[p]{f}:= \begin{cases} \displaystyle \left[\int_0^1 \abs{f(x)}^p \, \dif x \right]^{1/p}, & 1 \le p < \infty,\\[1ex]
\displaystyle  \sup_{0 \le x \le 1} \abs{f(x)}, & p=\infty,
\end{cases}
\\
\cv^{k}: =\cv^{k}[0,1]=\{f\in C[0,1]: \Var(f^{(k)}) < \infty \}, \\
\mathcal{W}^{k,p}=\mathcal{W}^{k,p}[0,1]=\{f\in C[0,1]: \|f^{(k)}\|_{p}<\infty\}.
\end{gather}
\end{subequations}
The stronger semi-norm is $\Fnorm{f}:=\Var(f')$, while the weaker semi-norm is
\[
\Ftnorm{f}:=\norm[1]{f'-A_2(f)'}=\norm[1]{f'-f(1)+f(0)}=\Var(f-A_2(f)),
\]
where $A_2(f): x \mapsto f(0)(1-x)+f(1)x$ is the linear interpolant of $f$ using the two endpoints of the integration interval. The reason for defining $\Ftnorm{f}$ this way is that $\Ftnorm{f}$ vanishes if $f$ is a linear function, and linear functions are integrated exactly by the trapezoidal rule.  The cone of integrands is defined as
\begin{equation}\label{coneinteg}
\cc_{\tau}:=\{f\in \cv^{1}:\Var(f')\leq\tau\|f'-f(1)+f(0)\|_1\}.
\end{equation}

The algorithm for approximating $\norm[1]{f'-f(1)+f(0)}$ is the $\tcf$-semi-norm of the linear spline, $A_n(f)$:
\begin{align}
\nonumber
\tF_n(f)&:=\Ftnorm{A_n(f)}=\bignorm[1]{A_n(f)'-A_2(f)'} \\
\label{1direst}
&=\sum_{i=1}^{n-1}\left|f(x_{i+1})-f(x_{i}) - \frac{f(1)-f(0)}{n-1}\right|.
\end{align}
The variation of the first derivative of the linear spline of $f$, i.e.,
\begin{equation} \label{Fnormalg}
F_n(f) :=\Var(A_n(f)') = (n-1)\sum_{i=1}^{n-2} \bigabs{f(x_i) - 2 f(x_{i+1})+f(x_{i+2})},
\end{equation}
provides a lower bound on $\Var(f')$ for $n \ge 3$, and can be used in the necessary condition that $f$ lies in $\cc_\tau$ as described in Remark \ref{neccondrem}.
The mean value theorem implies that
\begin{align*}
F_n(f) &= (n-1)\sum_{i=1}^{n-1} \bigabs{[f(x_{i+2}) - f(x_{i+1})] - [f(x_{i+1}) - f(x_{i})]} \\
&= \sum_{i=1}^{n-1} \abs{f'(\xi_{i+1}) - f'(\xi_{i})} \le \Var(f'),
\end{align*}
where $\xi_i$ is some point in $[x_i,x_{i+1}]$.

\subsection{Adaptive Algorithm and Upper Bound on the Cost}

Constructing the adaptive algorithm for integration requires an upper bound on the error of $T_n$ and a two-sided bound on the error of $\tF_n$.  Note that $\tF_{n}(f)$ never overestimates $\Ftnorm{f}$ because
\begin{align*}
\Ftnorm{f} & = \bignorm[1]{f'-A_2(f)'}
= \sum_{i=1}^{n-1} \int_{x_i}^{x_{i+1}} \abs{f'(x) - A_2(f)'(x)} \, \dif x \\
& \ge \sum_{i=1}^{n-1} \abs{\int_{x_i}^{x_{i+1}} [f'(x) - A_2(f)'(x)] \, \dif x}=\norm[1]{A_n(f)'-A_2(f)'} = \tF_n(f).
\end{align*}
Thus, $h_{-}(n):=0$ and $\fc_n=\tfc_n=1$.

To find an upper bound on $\Ftnorm{f}-\tF_{n}(f)$, note that
\begin{equation*}
\Ftnorm{f} - \tF_{n}(f) = \Ftnorm{f} - \bigabs{A_n(f)}_{\tcf} \le \bigabs{f-A_n(f)}_{\tcf} = \bignorm[1]{f' -A_n(f)'},
\end{equation*}
since $(f-A_n(f))(x)$ vanishes for $x=0,1$.  Moreover,
\begin{equation} \label{onenormfp}
\bignorm[1]{f' -A_n(f)'} = \sum_{i=1}^{n-1} \int_{x_i}^{x_{i+1}} \abs{f'(x) -(n-1)[f(x_{i+1})-f(x_i)]} \, \dif x.
\end{equation}
Now we bound each integral in the summation.  For $i=1, \ldots, n-1$, let $\eta_i(x) = f'(x) -(n-1)[f(x_{i+1})-f(x_i)]$, and let $p_i$ denote the probability that $\eta_i(x)$ is non-negative:
\[
p_i = (n-1)\int_{x_i}^{x_{i+1}} \bbone_{[0,\infty)} (\eta_i(x)) \, \dif x,
\]
and so $1-p_i$ is the probability that $\eta_i(x)$ is negative.  Since $\int_{x_i}^{x_{i+1}} \eta_i(x) \, \dif x =0$, we know that $\eta_i$ must take on both non-positive and non-negative values.  Invoking the mean value theorem, it follows that
\begin{multline*}
\frac{p_i}{n-1} \sup_{x_i \le x \le x_{i+1}} \eta_i(x) \ge \int_{x_i}^{x_{i+1}} \max(\eta_i(x),0) \, \dif x \\
= \int_{x_i}^{x_{i+1}} \max(-\eta_i(x),0) \, \dif x \le \frac{-(1-p_i)}{n-1} \inf_{x_i \le x \le x_{i+1}} \eta_i(x) .
\end{multline*}
These bounds allow us to derive bounds on the integrals in \eqref{onenormfp}:
\begin{align*}
\MoveEqLeft{\int_{x_i}^{x_{i+1}} \abs{\eta_i(x)} \, \dif x} \\
 &= \int_{x_i}^{x_{i+1}} \max(\eta_i(x),0) \, \dif x + \int_{x_i}^{x_{i+1}} \max(-\eta_i(x),0) \, \dif x\\
&=2(1-p_i) \int_{x_i}^{x_{i+1}} \max(\eta_i(x),0) \, \dif x + 2p_i\int_{x_i}^{x_{i+1}} \max(-\eta_i(x),0) \, \dif x\\
&\le \frac{2p_i(1-p_i)}{n-1} \left[ \sup_{x_i \le x \le x_{i+1}} \eta_i(x) - \inf_{x_i \le x \le x_{i+1}} \eta_i(x) \right]\\
&\le\frac{1}{2(n-1)} \left[ \sup_{x_i \le x \le x_{i+1}} f'(x) - \inf_{x_i \le x \le x_{i+1}} f'(x) \right],
\end{align*}
since $p_i(1-p_i)\le 1/4$.

Plugging this bound into \eqref{onenormfp} yields
\begin{align*}
\bignorm[1]{f'-f(1)+f(0)} - \tF_n(f) &= \Ftnorm{f} - \tF_{n}(f)\\
 & \le \bignorm[1]{f' -A_n(f)'}\\
&\le \frac{1}{2n-2} \sum_{i=1}^{n-1} \left[ \sup_{x_i \le x \le x_{i+1}} f'(x) - \inf_{x_i \le x \le x_{i+1}} f'(x) \right] \\
& \le \frac{\Var(f')}{2n-2} = \frac{\Fnorm{f}}{2n-2},
\end{align*}
and so
\begin{equation*}\label{factor}
h_{+}(n):= \frac{1}{2n-2}, \qquad \mathfrak{C}_n =\frac{1}{1 - \tau/(2n-2)} \qquad \text{for } n>1+\tau/2.
\end{equation*}
Since $\tF_2(f)=0$ by definition, the above inequality for $\Ftnorm{f} - \tF_{2}(f)$ implies that
\begin{equation*} \label{taumininteg}
2\bignorm[1]{f'-f(1)+f(0)} = 2 \Ftnorm{f} \le \Fnorm{f} = \Var(f'), \qquad \tau_{\min}=2.
\end{equation*}

The error of the trapezoidal rule in terms of the variation of the first derivative of the integrand is given in \cite[(7.15)]{BraPet11a}:
%\begin{subequations} \label{integhhtilde}
\begin{gather*}
\abs{\int_0^1 f(x) \, dx - T_n(f)} \le h(n) \Var(f') \\
h(n):= \frac{1}{8(n-1)^2}, \qquad h^{-1}(\varepsilon) = \left \lceil \sqrt{\frac{1}{8\varepsilon}} \right \rceil +1.
\end{gather*}
%\end{subequations}

Given the above definitions of $h, \fC_n, \fc_n$, and $\tfc_n$, it is now possible to also specify
\begin{subequations} \label{simplifycond}
\begin{gather}
h_1(n) = h_2(n) = \fC_n h(n) = \frac{1}{4(n-1)(2n-2-\tau)}, \\
h_1^{-1}(\varepsilon) = h_2^{-1}(\varepsilon) = 1+ \left \lceil \sqrt{\frac{\tau}{8 \varepsilon} + \frac{\tau^2}{16}} +\frac{\tau}{4} \right \rceil \le 2 + \frac{\tau}{2} + \sqrt{\frac{\tau}{8\varepsilon}}.
\end{gather}
Moreover, the left side of \eqref{multistageconv}, the stopping criterion inequality in the multi-stage algorithm, becomes
\begin{equation}
\tau h(n_i)\fC_{n_i} \tF_{n_i}(f) = \frac{\tau  \tF_{n_i}(f) } {4(n_i-1)(2n_i-2 -\tau)}.
\end{equation}
\end{subequations}
With these preliminaries, Algorithm \ref{multistagealgo} and Theorem \ref{MultiStageThm} may be applied directly to  yield the following adaptive integration algorithm and its guarantee.

\begin{algo}[Adaptive Univariate Integration] \label{multistageintegalgo}
Let the sequence of algorithms $\{A_n\}_{n\in \mathcal{I}}$, $\{\tF_n\}_{n\in \mathcal{I}}$, and $\{F_n\}_{n\in \mathcal{I}}$ be as described above.
Let $\tau\ge2$ be the cone constant. Set $i=1$. Let $n_1=\lceil(\tau+1)/2\rceil+1$. For any error tolerance $\varepsilon$ and input function $f$, do the following:
\begin{description}
\item[Stage 1.\ Estimate {$\norm[1]{f'-f(1)+f(0)}$} and bound {$\Var(f')$}.] Compute $\tF_{n_i}(f)$ in \eqref{1direst} and $F_{n_i}(f)$ in \eqref{Fnormalg}.

\item[Stage 2. Check the necessary condition for $f \in \cc_{\tau}$.] Compute
    \begin{align*}
     \tau_{\min,n_i} =  \frac{F_{n_i}(f)}{\tF_{n_i}(f)+F_{n_i}(f)/(2n_i-2)}.
    \end{align*}
If $\tau \ge \tau_{\min,n_i}$, then go to stage 3.  Otherwise, set $\tau = 2\tau_{\min,n_i}$.  If $n_i \ge (\tau+1)/2$, then go to stage 3.  Otherwise, choose
$$
n_{i+1}=1+ (n_i-1)\left\lceil\frac{\tau+1}{2n_i-2}\right\rceil.
$$
Go to Stage 1.

\item[Stage 3. Check for convergence.] Check whether $n_i$ is large enough to satisfy the error tolerance, i.e.
    \begin{equation*}
     \tF_{n_i}(f) \le \frac{4\varepsilon(n_i-1)(2n_i-2 - \tau)}{\tau}.
    \end{equation*}
If this is true, then return $T_{n_i}(f)$ and terminate the algorithm.   If this is not true, choose
$$
n_{i+1}=1+ (n_i-1)\max\left\{2,\left\lceil\frac{1}{(n_i-1)}\sqrt{\frac{\tau \tF_{n_i}(f)}{8\varepsilon}}\right\rceil\right\}.
$$
Go to Stage 1.
\end{description}
\end{algo}

\begin{theorem} \label{multistageintegthm}
Let $\sigma >0$ be some fixed parameter, and let $\cb_{\sigma}=\{f \in  \mathcal{V}^{1} : \Var(f')\leq \sigma\}$. Let $T \in \ca(\cb_{\sigma}, \reals, \INT, \Lambda^{\std})$ be the non-adaptive trapezoidal rule defined by Algorithm \ref{nonadaptalgo}, and let $\varepsilon>0$ be the error tolerance. Then this algorithm succeeds for $f \in \cb_{\sigma}$, i.e., $\abs{\INT(f) - T(f,\varepsilon)} \le \varepsilon$, and the cost of this algorithm is $\left \lceil \sqrt{\sigma/(8\varepsilon)}\right \rceil + 1$, regardless of the size of $\Var(f')$.

Now let $T \in \ca(\cc_{\tau}, \reals, \INT, \Lambda^{\std})$ be the adaptive trapezoidal rule defined by Algorithm \ref{multistageintegalgo}, and let $\tau$, $n_1$, and $\varepsilon$ be as described there. Let $\cc_\tau$ be the cone of functions defined in \eqref{coneinteg}.  Then it follows that Algorithm \ref{multistageintegalgo} is successful for all functions in $\cc_{\tau}$,  i.e.,  $\abs{\INT(f) - T(f,\varepsilon)} \le \varepsilon$.  Moreover, the cost of this algorithm is bounded below and above as follows:
\begin{multline}
\max \left(\left \lceil\frac{\tau+1}{2} \right \rceil, \left \lceil \sqrt{\frac{ \Var(f')}{8\varepsilon}} \right \rceil \right) +1 \\
\le \max \left(\left \lceil\frac{\tau+1}{2} \right \rceil, \left \lceil \sqrt{\frac{\tau \norm[1]{f'-f(1)+f(0)}}{8\varepsilon}} \right \rceil \right) +1 \\
\le
\cost(T,f;\varepsilon,N_{\max}) \\
\le \sqrt{\frac{\tau \norm[1]{f'-f(1)+f(0)}}{2\varepsilon}} + \tau + 4
\le \sqrt{\frac{\tau \Var(f') }{4\varepsilon}} + \tau + 4.
\end{multline}
The algorithm is computationally stable, meaning that the minimum and maximum costs for all integrands, $f$, with fixed $\norm[1]{f'-f(1)+f(0)}$ or $\Var(f')$ are an $\varepsilon$-independent constant of each other.
\end{theorem}

\subsection{Lower Bound on the Computational Cost}
Next, we derive a lower bound on the cost of approximating functions in the ball $\cb_{\sigma}$ and in the cone $\cc_{\tau}$ by constructing fooling functions. Following the arguments of Section \ref{LowBoundSec}, we choose  the triangle shaped function $f_0: x \mapsto 1/2-\abs{1/2-x}$. Then
\begin{gather*}
\Ftnorm{f_0}=\norm[1]{f'_0-f_0(1)+f_0(0)}=\int_0^1 \abs{\sign(1/2-x)} \, \dif x = 1, \\ \Fnorm{f_0}=\Var(f'_0)=2= \tau_{\min}.
\end{gather*}
For any $n \in \cj:=\natzero$, suppose that the one has the data $L_i(f)=f(\xi_i)$, $i=1, \ldots, n$ for arbitrary $\xi_i$, where $0=\xi_0 \le \xi_1 < \cdots < \xi_n \le \xi_{n+1} = 1$.  There must be some $j=0, \ldots, n$ such that $\xi_{j+1} - \xi_j \ge 1/(n+1)$.  The function $f_{1}$ is defined as a triangle function on the interval $[\xi_j, \xi_{j+1}]$:
$$
f_{1}(x):=\begin{cases} \displaystyle
\frac{\xi_{j+1}-\xi_{j}-\abs{\xi_{j+1}+\xi_{j}-2x}}{8} & \xi_{j} \le x \leq \xi_{j+1},\\
0 & \text{otherwise}.
\end{cases}
$$
This is a piecewise linear function whose derivative changes from $0$ to $1/4$ to $-1/4$ to $0$ provided $0 < \xi_j < \xi_{j+1} < 1$, and so $\Fnorm{f_1}=\Var(f'_1)\le 1$. Moreover,
\begin{gather*}
\INT(f)=\int_0^1 f_1(x) \, \dif x = \frac{(\xi_{j+1} - \xi_j)^2}{16} \ge \frac{1}{16(n+1)^2} =: g(n),\\
g^{-1}(\varepsilon)=\left \lceil \sqrt{\frac{1}{16 \varepsilon}} \right \rceil - 1.
\end{gather*}
Using these choices of $f_0$ and $f_1$, along with the corresponding $g$ above, one may invoke Theorems \ref{complowbdball}--\ref{complowbd}, and Corollary \ref{optimcor} to obtain the following theorem.

\begin{theorem} \label{complowbdinteg} For $\sigma>0$ let $\cb_{\sigma}=\{f \in \cv^{1} : \Var(f') \le \sigma\}$.  The complexity of integration on this ball is bounded below as
\begin{equation*}
\comp(\varepsilon,\ca(\cb_{\sigma},\reals,\INT,\Lambda^{\std}),\cb_{s}) \ge \left \lceil \sqrt{\frac{\min(s,\sigma)}{16 \varepsilon}} \right \rceil -1 .
\end{equation*}
Algorithm \ref{nonadaptalgo} using the trapezoidal rule has optimal order in the sense of Theorem \ref{optimalprop}.

For $\tau>2$, the complexity of the integration problem over the cone of functions $\cc_{\tau}$ defined in \eqref{coneinteg} is bounded below as
\begin{equation*}
\comp(\varepsilon,\ca(\cc_{\tau},\reals,\INT,\Lambda^{\std}),\cb_{s}) \ge \left \lceil \sqrt{\frac{(\tau-2)s}{32 \tau \varepsilon}} \right \rceil -1 .
\end{equation*}
The adaptive trapezoidal Algorithm \ref{multistageintegalgo} has optimal order for integration of functions in $\cc_{\tau}$ in the sense of Corollary \ref{optimcor}.
\end{theorem}

\subsection{Numerical Example} \label{integnumexamplesec}

Consider the family of bump test functions defined by
\begin{multline}\label{testfun}
f(x)= \\
\begin{cases}
\displaystyle  b[4a^2 + (x-z)^2-(x-z-a)|x-z-a|\\
\qquad \qquad -(x-z+a)|x-z+a|], & z-2a\leq x\leq z+2a,\\[2ex]
\displaystyle  0, & \text{otherwise}.
\end{cases}
\end{multline}
with  $\log_{10}(a) \sim \cu[-4,-1]$, $z \sim \cu[2a,1-2a]$, and $b=1/(4a^3)$ chosen to make $\int_0^1 f(x) \, \dif x = 1$.  It follows that $\norm[1]{f'-f(1)+f(0)}=1/a$ and $\Var(f')=2/a^2$.  The probability that $f \in \cc_{\tau}$ is $\min\left(1,\max(0,\left(\log_{10}(\tau/2)-1\right)/3)\right).$

As an experiment, we chose $10000$ random test functions and applied Algorithm \ref{multistageintegalgo} with an error tolerance of  $\varepsilon = 10^{-8}$ and initial $\tau$ values of $10, 100, 1000$.  The algorithm is considered successful for a particular $f$ if the exact and approximate integrals agree to within $\varepsilon$. The success and failure rates are given in Table \ref{integresultstable}. Our algorithm imposes a cost budget of $N_{\max}=10^7$.  If the proposed $n_{i+1}$ in Stages 2 or 3 exceeds $N_{\max}$, then our algorithm returns a warning and falls back to the largest possible $n_{i+1}$ not exceeding $N_{\max}$ for which $n_{i+1}-1$ is a multiple of $n_i-1$.  The probability that $f$ initially lies in $\cc_{\tau}$ is the smaller number in the third column of Table \ref{integresultstable}, while the larger number is the empirical probability that $f$ eventually lies in $\cc_{\tau}$ after possible increases in $\tau$ made by Stage 2 of Algorithm \ref{multistageintegalgo}.  For this experiment Algorithm \ref{multistageintegalgo} was successful for all $f$ that finally lie inside $\cc_{\tau}$, for which there was no warning.  It was also successful for a small percentage of functions lying outside the cone.

\begin{table}[h]
\centering
\begin{tabular}{cccccc}
&&&Success & Success & Failure \\
& $\tau$ &  $\Prob(f \in \cc_{\tau}) $ & No Warning & Warning & No Warning \\
\toprule
&$10$ & $0\% \rightarrow  25\% $ & $25\%$ & $<1\%$ & $75\%$  \\
Algorithm \ref{multistageintegalgo}
 &$100$ & $23 \% \rightarrow 58\% $ & $56\%$ & $2\%$ & $42\%$ \\
&$1000$ & $57\% \rightarrow 88\% $& $68\%$ & $20\%$ &$12\%$ \\
\midrule
{\tt quad} & & & 8\% & & $92\%$\\
{\tt integral} & & & 19\% & & $81\%$\\
{\tt chebfun} & & &29\% & & $71\%$\\
\end{tabular}
\caption{The probability of the test function lying in the cone for the original and eventual values of $\tau$ and the empirical success rate of Algorithm \ref{multistageintegalgo} plus the success rates of other common quadrature algorithms. \label{integresultstable}}
\end{table}

Some commonly available numerical algorithms in MATLAB are {\tt quad} and {\tt integral} \cite{MAT8.1} and the MATLAB Chebfun toolbox \cite{TrefEtal12}. We applied these three routines to the random family of test functions.  Their success and failure rates are also recorded in Table \ref{integresultstable}.  They do not give warnings of possible failure.

\section{$\cl_{\infty}$ Approximation of Univariate Functions} \label{approxsec}

Now we consider the problem of $\cl_{\infty}$ recovery of functions, i.e.,
\[
S(f):=\APP(f):=f, \qquad \cg:=\cl_{\infty}, \qquad \norm[\cg]{S(f)-A(f)}=\norm[\infty]{f-A(f)}.
\]
The space of functions to be recovered is the Sobolev space $\cf := \mathcal{W}^{2,\infty}$, as defined in \eqref{defSobolev}.  Our adaptive algorithm is defined on the following cone of functions
\begin{subequations} \label{coneappxdef}
\begin{gather}
\Ftnorm{f}:=\norm[\infty]{f'-f(1)+f(0)}, \qquad \Fnorm{f}:=\norm[\infty]{f''}, \\
\cc_{\tau}:=\{f \in  \mathcal{W}^{2,\infty} : \norm[\infty]{f''}\leq \tau\norm[\infty]{f'-f(1)+f(0)}\}.
\end{gather}
\end{subequations}

The basic fixed-cost algorithm used to approximate functions is the linear spline algorithm given in \eqref{linearspline}.
The cost of $A_n$ is $n$, and the cost multiple is $r=2$.
Using this same data one may approximate the $\cl_\infty$ norm of $f'-f(1)+f(0)$ by the algorithm
\begin{multline}\label{estfirstderiv}
\tF_n(f) := \norm[\infty]{A_n(f)' - A_2(f)'} \\
= \sup_{i=1, \ldots, n-1} \bigabs{(n-1) [f(x_{i+1})-f(x_i)]-f(1)+f(0)}.
\end{multline}
Moreover, a lower bound on $\norm[\infty]{f''}$ can be derived similarly to the previous section using a centered difference.  Specifically, for $n \ge 3$,
\begin{equation} \label{Fnormappxalg}
F_n(f) := (n-1)^2\sup_{i=1, \ldots, n-2} \abs{f(x_i) - 2 f(x_{i+1})+f(x_{i+2})}.
\end{equation}
It follows using the H\"older's inequality that
\begin{align*}
F_n(f) &= (n-1)^2\sup_{i=1, \ldots, n-2} \biggabs{\int_{x_i}^{x_{i+2}} \left[\frac{1}{n-1} - \abs{x-x_{i+1}}\right] f''(x) \, \dif x} \\
&\le (n-1)^2\sup_{i=1, \ldots, n-2} \norm[\infty]{f''} \int_{x_i}^{x_{i+2}} \abs{\frac{1}{n-1} - \abs{x-x_{i+1}}} \, \dif x = \norm[\infty]{f''}.
\end{align*}

\subsection{Adaptive Algorithm and Upper Bound on the Cost}

Given the algorithms $\tF_n$ and $A_n$, we now turn to deriving the worst case error bounds, $h_{\pm}$ defined in \eqref{Gerrbds} and $h$ defined in \eqref{algoerr} and satisfying \eqref{algseqerrcond} for $\ci:=\{2, 3, \ldots\}$.
Note that $\tF_{n}(f)$ never overestimates $\Ftnorm{f}$ because
\begin{align*}
\Ftnorm{f} &  = \bignorm[\infty]{f'-A_2(f)'}
= \sup_{\substack{x_i \le x \le x_{i+1}\\ i=1, \ldots, n-1 }} \abs{f'(x) - A_2(f)'(x)} \\
&\ge \sup_{i=1, \ldots, n-1} (n-1)\int_{x_i}^{x_{i+1}} \abs{f'(x) - f(1)+f(0)} \, \dif x \\
&\ge \sup_{i=1, \ldots, n-1} (n-1)\abs{\int_{x_i}^{x_{i+1}} [f'(x) - f(1)+f(0)] \, \dif x }\\
&= \sup_{i=1, \ldots, n-1} (n-1)\abs{f(x_{i+1})-f(x_i) - \frac{f(1)-f(0)}{n-1}} = \tF_n(f).
\end{align*}
Thus, $h_{-}(n):=0$ and $\fc_n=\tfc_n=1$.

The difference between $f$ and its linear spline can be bounded in terms of an integral involving the second derivative using integration by parts.  For $x \in [x_i,x_{i+1}]$ it follows that
\begin{align}
\nonumber
f(x)-A_n (f)(x)
&= f(x) - (n-1) \left[f(x_{i})(x_{i+1}-x) +f(x_{i+1})(x-x_i) \right]\\
& = (n-1) \int_{x_i}^{x_{i+1}} v_i(t,x) f''(t) \, \dif t, \label{fintermsfdoubprime}\\
f'(x)-A_n(f)'(x) & = (n-1) \int_{x_i}^{x_{i+1}} \frac{\partial v_i}{\partial x}(t,x) f''(t) \, \dif t, \label{fprimeintermsfdoubprime}
\end{align}
where the continuous, piecewise differentiable kernel $v$ is defined as
\begin{equation*}
v_i(t,x) :=\begin{cases} (x_{i+1}-x)(x_i-t), & x_i\leq t\leq x,\\
(x-x_{i})(t- x_{i+1}), & x< t \leq x_{i+1},
\end{cases}.
\end{equation*}

To find an upper bound on $\Ftnorm{f}-\tF_{n}(f)$, note that
\begin{equation*}
\Ftnorm{f} - \tF_{n}(f) = \Ftnorm{f} - \bigabs{A_n(f)}_{\tcf} \le \bigabs{f-A_n(f)}_{\tcf} = \bignorm[\infty]{f' -A_n(f)'},
\end{equation*}
since $(f-A_n(f))(x)$ vanishes for $x=0,1$. Using \eqref{fprimeintermsfdoubprime} it then follows that
\begin{align*}
\Ftnorm{f} - \tF_{n}(f) & \le \bignorm[\infty]{f' -A_n(f)'} \\
& = \sup_{\substack{x_i \le x \le x_{i+1}\\ i=1, \ldots, n-1 }} \abs{f'(x) -(n-1)[f(x_{i+1})-f(x_i)]} \, \dif x \\
&=(n-1) \sup_{\substack{x_i \le x \le x_{i+1}\\ i=1, \ldots, n-1 }} \abs{ \int_{x_i}^{x_{i+1}} \frac{\partial v_i}{\partial x}(t,x) f''(t) \, \dif t} \\
& \le (n-1) \norm[\infty]{f''} \sup_{\substack{x_i \le x \le x_{i+1}\\ i=1, \ldots, n-1 }} \int_{x_i}^{x_{i+1}} \abs{\frac{\partial v_i}{\partial x}(t,x)} \, \dif t \\
&=(n-1)\norm[\infty]{f''} \sup_{\substack{x_i \le x \le x_{i+1}\\ i=1, \ldots, n-1 }} \left \{\frac{1}{2(n-1)^2} - (x-x_i)(x_{i+1}-x) \right\} \\
&= h_+(n) \norm[\infty]{f''}, \qquad \qquad  h_+(n):= \frac{1}{2(n-1)}.
\end{align*}
This implies that $\fC_n =1/[1 - \tau /(2n-2)]$ provided that $n>1+\tau/2$.
Since $\tF_2(f)=0$ by definition, the above inequality for $\Ftnorm{f} - \tF_{2}(f)$ implies that $\tau_{\min}=2$.

To derive the error bounds for $A_n(f)$ we make use of \eqref{fintermsfdoubprime}:
\begin{align*}
\norm[\infty]{f-A_n(f)}
& \le \sup_{\substack{x_i \le x \le x_{i+1}\\ i=1, \ldots, n-1 }} \abs{f(x)-A_n (f)(x)}\\
&= (n-1)\sup_{\substack{x_i \le x \le x_{i+1}\\ i=1, \ldots, n-1 }}  \int_{x_i}^{x_{i+1}} \abs{v_i(t,x) f''(t)} \, \dif t\\
&= (n-1)\norm[\infty]{f''} \sup_{\substack{x_i \le x \le x_{i+1}\\ i=1, \ldots, n-1 }}  \int_{x_i}^{x_{i+1}} \abs{v_i(t,x)} \, \dif t\\
&=\norm[\infty]{f''} \sup_{\substack{x_i \le x \le x_{i+1}\\ i=1, \ldots, n-1 }}  \frac{(x-x_i)(x_{i+1}-x)}2 \\
&= h(n) \norm[\infty]{f''}, \qquad \qquad  h(n):= \frac{1}{8(n-1)^2}.
\end{align*}
Since $h_{\pm}(n)$ and $h(n)$, are the same as in the previous section for integration, the simplifications in \eqref{simplifycond} apply here as well.  Then Algorithm \ref{multistagealgo} and Theorem \ref{MultiStageThm} may be applied directly to  yield the following algorithm for function approximation and its guarantee.

\begin{algo}[Adaptive Univariate Function Recovery] \label{multistageapproalgo}
Let the sequence of algorithms $\{A_n\}_{n\in \mathcal{I}}$, $\{\tF_n\}_{n\in \mathcal{I}}$, and $\{F_n\}_{n\in \mathcal{I}}$ be as described above.
Let $\tau \ge 2$ be the cone constant. Set $i=1$. Let $n_1=\lceil(\tau+1)/2\rceil+1$. For any error tolerance $\varepsilon$ and input function $f$, do the following:
\begin{description}
\item[Stage 1.\ Estimate {$\norm[\infty]{f'-f(1)+f(0)}$} and bound {$\norm[\infty]{f''}$}.] Compute $\tF_{n_i}(f)$ in \eqref{estfirstderiv} and $F_{n_i}(f)$ in \eqref{Fnormappxalg}.

\item[Stage 2. Check the necessary condition for $f \in \cc_{\tau}$.] Compute
    \begin{align*}
     \tau_{\min,n_i} =  \frac{F_{n_i}(f)}{\tF_{n_i}(f)+F_{n_i}(f)/(2n_i-2)}.
    \end{align*}
If $\tau \ge \tau_{\min,n_i}$, then go to stage 3.  Otherwise, set $\tau = 2\tau_{\min,n_i}$.  If $n_i \ge (\tau+1)/2$, then go to stage 3.  Otherwise, choose
$$
n_{i+1}=1+ (n_i-1)\left\lceil\frac{\tau+1}{2n_i-2}\right\rceil.
$$
Go to Stage 1.

\item[Stage 3. Check for convergence.] Check whether $n_i$ is large enough to satisfy the error tolerance, i.e.
    \begin{equation*}
     \tF_{n_i}(f) \le \frac{4\varepsilon(n_i-1)(2n_i-2 - \tau)}{\tau}.
    \end{equation*}
If this is true, then return $A_{n_i}(f)$ and terminate the algorithm.   If this is not true, choose
$$
n_{i+1}=1+ (n_i-1)\max\left\{2,\left\lceil\frac{1}{(n_i-1)}\sqrt{\frac{\tau \tF_{n_i}(f)}{8\varepsilon}}\right\rceil\right\}.
$$
Go to Stage 1.
\end{description}
\end{algo}

\begin{theorem} \label{multistageappxthm}
Let $\sigma >0$ be some fixed parameter, and let $\cb_{\sigma}=\{f \in  \mathcal{W}^{2,\infty} : \norm[\infty]{f''}\leq \sigma\}$. Let $A \in \ca(\cb_{\sigma}, \cl_{\infty}, \APP, \Lambda^{\std})$ be the non-adaptive linear spline defined by Algorithm \ref{nonadaptalgo}, and let $\varepsilon>0$ be the error tolerance. Then this algorithm succeeds for $f \in \cb_{\sigma}$, i.e., $\norm[\infty]{f - A(f,\varepsilon)} \le \varepsilon$, and the cost of this algorithm is $\left \lceil \sqrt{\sigma/(8\varepsilon)}\right \rceil + 1$, regardless of the size of $\norm[\infty]{f''}$.

Let $A \in \ca(\cc_{\tau}, \cl_{\infty}, \APP, \Lambda^{\std})$ be the adaptive linear spline defined by Algorithm \ref{multistageapproalgo}, and let $\tau$, $n_1$, and $\varepsilon$ be the inputs and parameters described there. Let $\cc_\tau$ be the cone of functions defined in \eqref{coneappxdef}.  Then it follows that Algorithm \ref{multistageapproalgo} is successful for all functions in $\cc_{\tau}$,  i.e.,  $\norm[\infty]{f - A(f,\varepsilon)} \le \varepsilon$.  Moreover, the cost of this algorithm is bounded below and above as follows:
\begin{multline}
\max \left(\left \lceil\frac{\tau+1}{2} \right \rceil, \left \lceil \sqrt{\frac{ \norm[\infty]{f''}}{8\varepsilon}} \right \rceil \right) +1 \\
\le \max \left(\left \lceil\frac{\tau+1}{2} \right \rceil, \left \lceil \sqrt{\frac{\tau \norm[\infty]{f'-f(1)+f(0)}}{8\varepsilon}} \right \rceil \right) +1 \\
\le
\cost(A,f;\varepsilon,N_{\max}) \\
\le \sqrt{\frac{\tau \norm[\infty]{f'-f(1)+f(0)}}{2\varepsilon}} + \tau + 4
\le \sqrt{\frac{\tau \norm[\infty]{f''} }{4\varepsilon}} + \tau + 4.
\end{multline}
The algorithm is computationally stable, meaning that the minimum and maximum costs for all integrands, $f$, with fixed $\norm[\infty]{f'-f(1)+f(0)}$ or $\norm[\infty]{f''}$ are an $\varepsilon$-independent constant of each other.
\end{theorem}

\subsection{Lower Bound on the Computational Cost}
Next, we derive a lower bound on the cost of approximating functions in the ball $\cb_{\tau}$ and in the cone $\cc_{\tau}$ by constructing fooling functions. Following the arguments of Section \ref{LowBoundSec}, we choose the parabola $f_0: x \mapsto x(1-x)$. Then
\begin{gather*}
\Ftnorm{f_0}=\norm[\infty]{f'_0-f_0(1)+f_0(0)}=\sup_{0 \le x \le 1} \abs{1-2x} = 1, \\ \Fnorm{f_0}=\norm[\infty]{f''_0}=2= \tau_{\min}.
\end{gather*}
For any $n \in \cj:=\natzero$, suppose that the one has the data $L_i(f)=f(\xi_i)$, $i=1, \ldots, n$ for arbitrary $\xi_i$, where $0=\xi_0 \le \xi_1 < \cdots < \xi_n \le \xi_{n+1} = 1$.  There must be some $j=0, \ldots, n$ such that $\xi_{j+1} - \xi_j \ge 1/(n+1)$.  The function $f_{1}$ is defined as a bump having piecewise constant second derivative on $[\xi_j, \xi_{j+1}]$ and zero elsewhere.  For $\xi_{j} \le x \leq \xi_{j+1}$,
\begin{multline*}
f_{1}(x):=
\frac{1}{32} \left [4(\xi_{j+1}-\xi_j)^2 + (4x-2\xi_j-2\xi_{j+1})^2  \right. \\
\left. + (4x-\xi_j-3\xi_{j+1})\abs{4x-\xi_j-3\xi_{j+1}} -(4x-3\xi_j-\xi_{j+1})\abs{4x-3\xi_j-\xi_{j+1}} \right],
\end{multline*}
\[
f'_{1}(x)=
\frac{1}{4} \left [4x-2\xi_j-2\xi_{j+1} + \abs{4x-\xi_j-3\xi_{j+1}} -\abs{4x-3\xi_j-\xi_{j+1}} \right],
\]
\[
f''_{1}(x)=\sgn(4x-\xi_j-3\xi_{j+1}) - \sgn(4x-3\xi_j-\xi_{j+1}) + 1.
\]
This bump function is similar to the one used in the numerical examples in the previous section and this section.  For this bump $\norm[\infty]{f''_1}=1$, and
\[
\norm[\infty]{f_1}=f_1((\xi_j+\xi_{j+1})/2)= \frac{(\xi_{j+1} - \xi_j)^2}{16} \ge \frac{1}{16(n+1)^2} =: g(n).
\]
Using these choices of $f_0$ and $f_1$, along with the corresponding $g$ above, one may invoke Theorems \ref{complowbdball}--\ref{complowbd}, and Corollary \ref{optimcor} to obtain the following theorem.

\begin{theorem} \label{complowbdappr} For $\sigma>0$ let $\cb_{\sigma}=\{f \in \cw^{2,\infty} : \norm[\infty]{f''} \le \sigma\}$.  The complexity of function recovery on this ball is bounded below as
\begin{equation*}
\comp(\varepsilon,\ca(\cb_{\sigma},\cl_\infty,\APP,\Lambda^{\std}),\cb_{s}) \ge \left \lceil \sqrt{\frac{\min(s,\sigma)}{16 \varepsilon}} \right \rceil -1 .
\end{equation*}
Algorithm \ref{nonadaptalgo} using linear splines has optimal order in the sense of Theorem \ref{optimalprop}.

For $\tau>2$, the complexity of the function recovery problem over the cone of functions $\cc_{\tau}$ defined in \eqref{coneappxdef} is bounded below as
\begin{equation*}
\comp(\varepsilon,\ca(\cc_{\tau},\cl_\infty,\APP,\Lambda^{\std}),\cb_{s}) \ge \left \lceil \sqrt{\frac{(\tau-2)s}{32 \tau \varepsilon}} \right \rceil-1 .
\end{equation*}
The adaptive linear spline Algorithm \ref{multistageapproalgo} has optimal order for recovering functions in $\cc_{\tau}$ the sense of Corollary \ref{optimcor}.
\end{theorem}

\subsection{Numerical Example}

To illustrate Algorithm \ref{multistageapproalgo} we choose the same  family of test functions as in \eqref{testfun}, but now with $b=1/(2a^2)$.   Since $\norm[\infty]{f'-f(0)+f(1)}=1/a$ and $\norm[\infty]{f''}=1/a^2$, the probability that $f \in \cc_{\tau}$ is $\min\left(1,\max(0,\left(\log_{10}(\tau)-1\right)/3)\right).$
The number of random functions chosen, the error tolerance, the initial $\tau$ values, and the cost budget are the same as in Section \ref{integnumexamplesec}.  Table \ref{approxnumerical} shows results that are analogous to Table \ref{integresultstable}.  Algorithm \ref{multistageapproalgo} yields the correct value to within the error tolerance for all $f$ that finally lie inside $\cc_{\tau}$ and for which the algorithm does not try to exceed the cost budget.

\begin{table}[h]
\centering
\begin{tabular}{cccccc}
&&Success & Success & Failure & Failure \\
$\tau$ &  $\Prob(f \in \cc_{\tau}) $ & No Warning & Warning & No Warning & Warning\\
\toprule
$10$ & $0\% \rightarrow  26\% $ & $26\%$  & $<1\%$  &$74\%$ & $<1\%$\\
$100$ & $33 \% \rightarrow 57\% $ & $56\%$ & $1\%$ & $43\%$ & $1\%$\\
$1000$ & $67\% \rightarrow 88\% $& $75\%$ & $5\%$ & $12\%$ & $8\%$\\
\end{tabular}
\caption{The probability of the test function lying in the cone for the original and eventual values of $\tau$ and the empirical success rate of Algorithm \ref{multistageapproalgo}.  \label{approxnumerical}}
\end{table}

\section{Addressing Questions and Concerns About Adaptive, Automatic Algorithms} \label{overcomesec}

Adaptive, automatic algorithms are popular, especially for univariate integration problems.  Several general purpose numerical computing environments have one or more automatic integration routines, such as MATLAB \cite{TrefEtal12,MAT8.1} and the NAG \cite{NAG23} library.  In spite of their popularity there remain important questions and concerns regarding adaptive algorithms.  This section attempts to address them.

\subsection{All Automatic Algorithms Can Be Fooled}  

Any algorithm that solves a problem involving an infinite-dimensional space of input functions can be fooled by a spiky function, i.e., one that yields zero data where probed by the algorithm, but is nonzero elsewhere.   Figure \ref{fig:foolquad}a) depicts a spiky integrand whose integral is $\approx 0.3694$, but for which MATLAB's {\tt quad}, which is based on adaptive Simpson's rule  \cite{GanGau00a}, gives the answer $0$, even with an error tolerance of $10^{-14}$.  Our criticism of algorithms like {\tt quad} is not that they can be fooled, but that there is no available theory to tell us what is wrong with the integrand when they are fooled.  Guaranteed algorithms specify conditions that rule out spiky functions that might fool these algorithms.  Non-adaptive algorithms such as Algorithm \ref{nonadaptalgo} require that input functions lie in a ball, while adaptive algorithms, such as  Algorithms \ref{twostagedetalgo} and \ref{multistagealgo}, require that input functions lie in a cone.

\begin{figure}
\centering 
\begin{tabular}{cc}
\includegraphics[width=5.5cm]{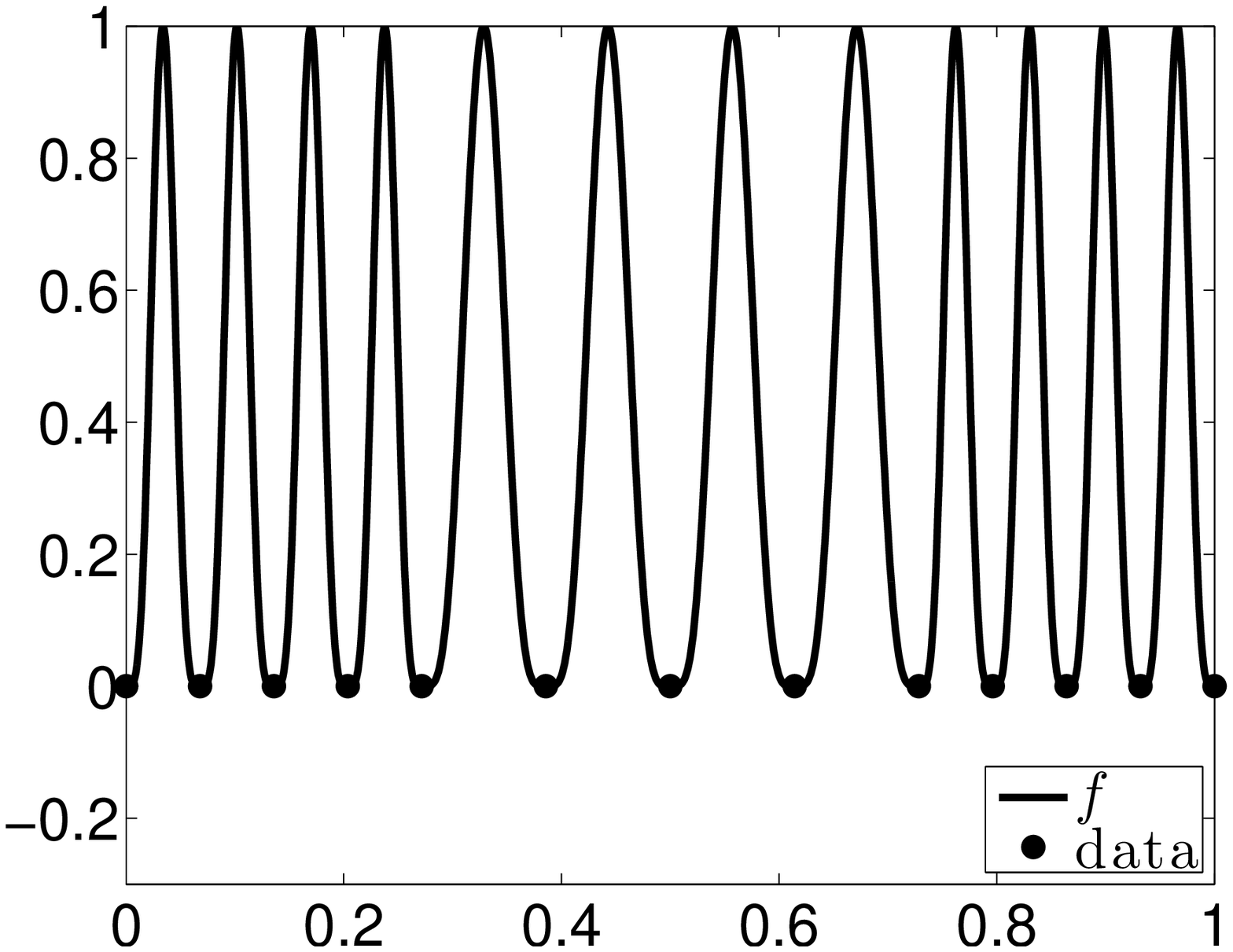}
&
\includegraphics[width=5.5cm]{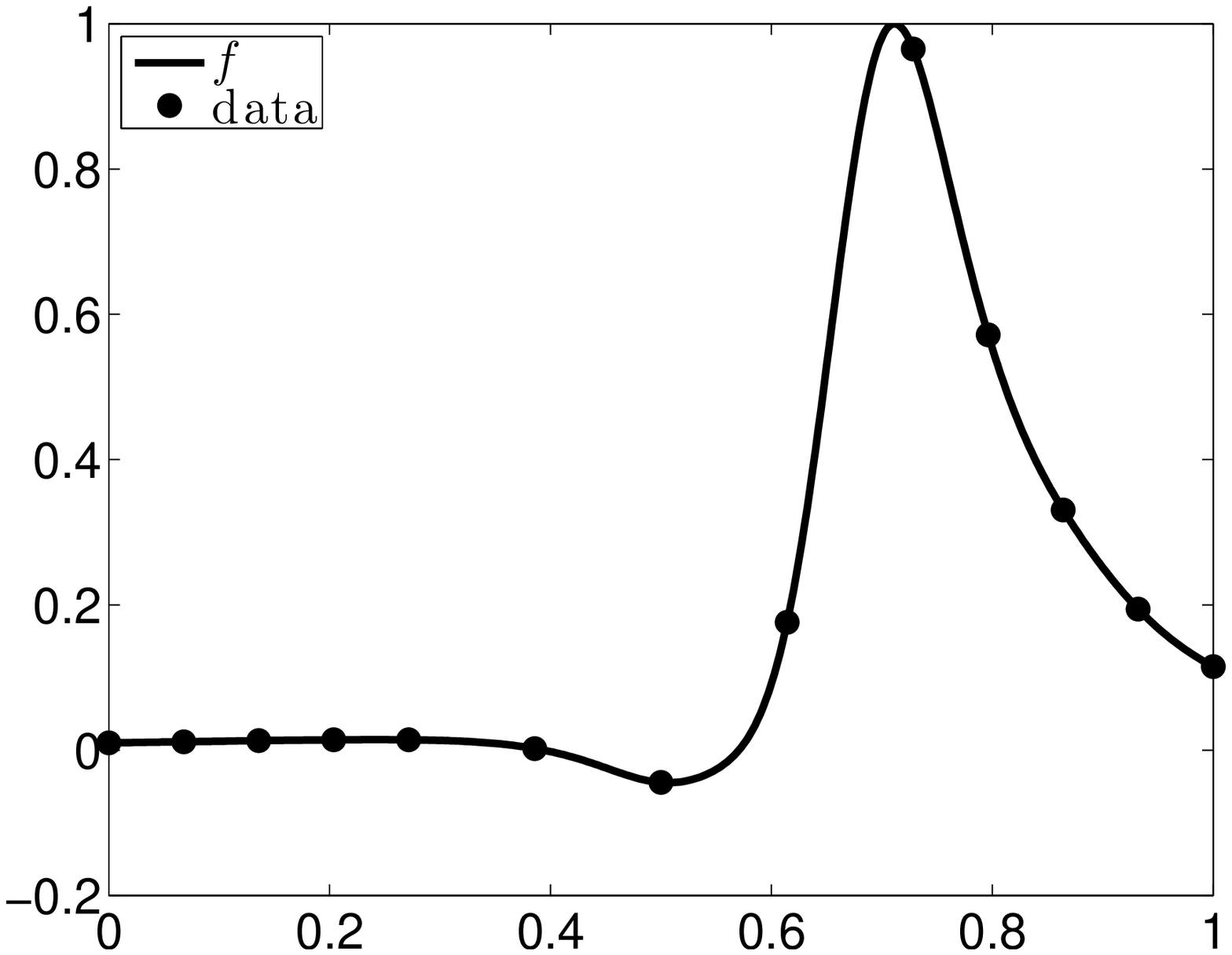} \\
a) & b)
\end{tabular}
\caption{a) A spiky integrand designed to fool MATLAB's {\tt quad} and the data sampled by {\tt quad}; b) A fluky integrand designed to fool {\tt quad}. \label{fig:foolquad}}
\end{figure}

\subsection{Why Cones?}

Most existing numerical analysis is focused on balls of input functions, $\cb_{\sigma}$, and the automatic algorithms arising from this analysis are non-adaptive, automatic such as Algorithm \ref{nonadaptalgo}. The analysis here focuses on cones of input functions, $\cc_{\tau}$, which allows us to derive data-driven error bounds and construct adaptive, automatic algorithms. We have two reasons for favoring cones.  

Since the solution operator, $S$, and the fixed-cost algorithms, $\{A_n\}_{n\in \ci}$, commonly encountered in practice are positively homogeneous, the error functional, $\err_n(\cdot) = \Gnorm{S(\cdot)-A_n(\cdot)}$ is also positively homogeneous.  This naturally suggests data-driven error bounds, $\herr_n(\cdot)$, that are positively homogeneous.  If $\err_n(f) \le \herr_n(f)$, then $\err_n(cf) \le \herr_n(cf)$ for $c\ge 0$. This leads us to consider cones of input functions.

A second reason to favor cones is that we want to spend less effort solving problems for input functions that are ``easy'', i.e., we want an adaptive algorithm.  At the end of Section \ref{coneoptsubsec} it was noted that our adaptive algorithms possess a stronger optimality than the non-adaptive one.  In particular, in Theorems \ref{multistageintegthm} and \ref{multistageappxthm} the costs of the non-adaptive algorithms do not depend on the norms of the input functions, but the costs of the adaptive algorithms do so in a favorable way.

There are rigorous results from information based complexity theory giving general conditions under which adaptive algorithms have no significant advantage over non-adaptive algorithms (e.g., see \cite[Chapter 4, Theorem 5.2.1]{TraWasWoz88} and \cite{Nov96a}). For adaption to be useful, we must violate one of these conditions.  In particular, we violate the condition that the set of input functions be convex.

To see why $\cc_{\tau}$ is not convex, let $f_{\text{in}}$ and $f_{\text{out}}$ be functions in $\cf$ with nonzero $\tcf$-semi-norms, where $f_{\text{in}}$  lies in the interior of this cone, and $f_{\text{out}}$ lies outside the cone.  This means that 
\[
\frac{\Fnorm{f_{\text{in}}}} {\Ftnorm{f_{\text{in}}}} = \tau_{\text{in}} < \tau < \tau_{\text{out}} =  \frac{\Fnorm{f_{\text{out}}}} {\Ftnorm{f_{\text{out}}}}.
\]
Next define two functions $f_{\pm} = (\tau-\tau_{\text{in}}) \Ftnorm{f_{\text{in}}} f_{\text{out}}  \pm (\tau + \tau_{\text{out}}) \Ftnorm{f_{\text{out}}} f_{\text{in}}$.
Since 
\begin{multline*} 
\Fnorm{f_{\pm}} \le (\tau-\tau_{\text{in}}) \Ftnorm{f_{\text{in}}} \Fnorm{f_{\text{out}}}  + (\tau + \tau_{\text{out}}) \Ftnorm{f_{\text{out}}} \Fnorm{f_{\text{in}}}\\
= [\tau_{\text{out}} (\tau-\tau_{\text{in}})  + \tau_{\text{in}} (\tau + \tau_{\text{out}})] \Ftnorm{f_{\text{in}}} \Ftnorm{f_{\text{out}}} =  \tau (\tau_{\text{out}} +\tau_{\text{in}}) \Ftnorm{f_{\text{in}}} \Ftnorm{f_{\text{out}}},
\end{multline*}
and 
\begin{multline*} 
\Ftnorm{f_{\pm}} \ge -(\tau-\tau_{\text{in}}) \Ftnorm{f_{\text{in}}} \Ftnorm{f_{\text{out}}}  + (\tau + \tau_{\text{out}}) \Ftnorm{f_{\text{out}}} \Ftnorm{f_{\text{in}}}\\
 =  (\tau_{\text{out}} +\tau_{\text{in}}) \Ftnorm{f_{\text{in}}} \Ftnorm{f_{\text{out}}},
\end{multline*}
it follows that $\Fnorm{f_{\pm}} \le \tau \Ftnorm{f_{\pm}}$, and so $f_{\pm} \in \cc_{\tau}$.  On the other hand $(f_- + f_+)/2$, which is a convex combination of $f_{+}$ and $f_-$, is $(\tau-\tau_{\text{in}}) \Ftnorm{f_{\text{in}}} f_{\text{out}}$.  Since $\tau > \tau_{\text{in}}$, this is a nonzero multiple of $f_{\text{out}}$, and it lies outside $\cc_{\tau}$.  Thus, this cone is not convex.

\subsection{Adaptive Algorithms that Stop When $A_{n_{i}}(f)-A_{n_{i-1}}(f)$ Is Small} \label{Lynesssubsec}

Many practical adaptive, automatic algorithms, especially those for univariate integration, are based on a stopping rule that returns $A_{n_{i}}(f)$ as the answer for the first $i$ where $\norm[\cg]{A_{n_{i}}(f)-A_{n_{i-1}}(f)}$ is small enough.  Fundamental texts in numerical algorithms advocate such stopping rules, e.g. \cite[p.\ 223--224]{BurFai10}, \cite[p.\ 233]{CheKin12a}, and \cite[p.\ 270]{Sau12a}.  Unfortunately, such stopping rules are problematic.

For instance, consider the univariate integration problem and the trapezoidal rule algorithm, $T_{n_i}$, based on $n_i=2^i+1$ points, i.e., $n_i-1=2^i$ trapezoids.  It is taught that the trapezoidal rule has the following error estimate:
\begin{equation} \label{traperrest}
\herr_i(f):=\frac{T_{n_{i}}(f)-T_{n_{i-1}}(f)}{3} \approx \int_0^1 f(x) \, \dif x - T_{n_i}(f) =:\err_i(f).
\end{equation}
Since $T_{n_i}(f)+ \herr_i(f)$ is exactly Simpson's rule, it follows that $\err_i(f) -\herr_i(f) = \Theta(16^{-i} \Var(f^{(3)}))$.  The error estimate may be good for moderate $i$, but it can only be guaranteed with some a priori knowledge of $\Var(f^{(3)})$.

In his provocatively titled SIAM Review article, \emph{When Not to Use an Automatic Quadrature Routine} \cite[p.\ 69]{Lyn83}, James Lyness makes the following claim.
\begin{quote}
While prepared to take the risk of being misled by chance alignment of zeros in the integrand function, or by narrow peaks which are ``missed,'' the user may wish to be reassured that for ``reasonable'' integrand functions which do not have these characteristics all will be well. It is the purpose of the rest of this section to demonstrate by example that he cannot be reassured on this point. In fact the routine is likely to be unreliable in a significant proportion of the problems it faces (say $1$ to $5\%$) and there is no way of predicting in a straightforward way in which of any set of apparently reasonable problems this will happen.
\end{quote}

Lyness's argument, with its pessimistic conclusion, is correct for commonly used adaptive, automatic algorithms.  Figure \ref{fig:foolquad}b depicts an integrand inspired by \cite{Lyn83} that we would call ``fluky''.  MATLAB's {\tt quad} gives the answer $\approx 0.1733$, for an absolute error tolerance of $\varepsilon=10^{-14}$, but the true answer is $\approx 0.1925$.  The {\tt quad} routine splits the interval of integration into three separate intervals and initially calculates Simpson's rule with one and two parabolas for each of the three intervals.  The data taken are denoted by $\bullet$ in Figure \ref{fig:foolquad}.  Since this fluky integrand is designed so that the two Simpson's rules match exactly for each of the three intervals, {\tt quad} is fooled into thinking that it knows the correct value of the integral and terminates immediately.

Lyness's warning in \cite{Lyn83} is a valid objection to commonly used stopping criteria based on a simple measure of the difference between two successive fixed-cost algorithms, e.g., error estimate \eqref{traperrest}. However, it is not a valid objection to adaptive, automatic algorithms in general, and it does not apply to our adaptive algorithms.

\section{Discussion and Further Work} \label{furthersec}

We believe that there should be more adaptive, automatic algorithms with rigorous guarantees of their success.  Users ought to be able to integrate functions, approximate functions, optimize functions, etc., without needing to manually tune the sample size.  Here we have shown how this might be done in general, as well as specifically for two case studies.  We hope that this will inspire further research in this direction.

The results presented here suggest a number of interesting open problems, some of which we are working on.  Here is a summary.

\begin{itemize}

\item This analysis should be extended to \emph{relative} error tolerances.  

\item The algorithms in Sections \ref{integsec} and \ref{approxsec} have low order convergence.  Guaranteed adaptive algorithms with \emph{higher order convergence rates} for smoother input functions are needed.

\item Other types of problems, e.g., linear differential equations and nonlinear optimization, fit the general framework presented here.  These problems have adaptive, automatic algorithms, but until now without guarantees.

\item The algorithms developed here are \emph{globally adaptive}, in the sense that the function data determines the sample size, but does not lead to denser sampling in areas of interest.  Since local adaption seems beneficial in practice, we need to develop such algorithms with guarantees.

\item For some numerical problems the error bound of the fixed-cost algorithm involves an $\tcf$- or $\cf$-semi-norm that is difficult to approximate.  An example is multivariate quadrature using quasi-Monte Carlo algorithms, where the error depends on the \emph{variation} of the multivariate integrand.  To obtain guaranteed automatic, adaptive algorithms one must either find an efficient way to approximate the semi-norm or find other suitable error bounds that can be reliably obtained from the function data.  

\item This article considers only the worst case error of deterministic algorithms.  Random algorithms must be analyzed by somewhat different methods.  A guaranteed Monte Carlo algorithm for estimating the mean of a random variable, which includes multivariate integration as a special case, has been proposed in \cite{HicEtal14a}.

\item Some of the authors and their collaborators are implementing the algorithms described here, along with others, in the open-source Guaranteed Automatic Integration Library (GAIL) for MATLAB (see \url{https://code.google.com/p/gail/}).  This library will also contain scripts that generate the tables and figures in this paper.

\end{itemize}

\section{Acknowledgements}  The authors are grateful to the editor and two referees for their valuable suggestions.  We are also grateful for fruitful discussions with a number of colleagues. This research is supported in part by grant NSF-DMS-1115392.

\section*{References}
\bibliographystyle{model1b-num-names.bst}
\bibliography{FJH22,FJHown22}
\end{document}